\documentclass[11pt,reqno]{amsart}

\usepackage{amscd,amssymb,amsmath,amsthm}
\usepackage[arrow,matrix]{xy}
\usepackage{graphicx}
\usepackage{color}
\usepackage{cite}
\usepackage{enumerate}
\usepackage{hyperref}
\usepackage{tikz}

\newtheorem{thm}{Theorem}[section]
\newtheorem{defn}[thm]{Definition}
\newtheorem{lem}[thm]{Lemma}
\newtheorem{pro}[thm]{Proposition}
\newtheorem{rk}[thm]{Remark}
\newtheorem{cor}[thm]{Corollary}

\newtheorem{con}[thm]{Condition}

\numberwithin{equation}{section} 

\DeclareMathOperator{\E}{\mathrm{E}}

\DeclareMathOperator{\asinh}{\mathrm{asinh}}
\DeclareMathOperator{\sgn}{\mathrm{sgn}}

\def\s{\sigma}

\def\r{\rho}

\def\s{\sigma}

\def\l{\lambda}

\def\d{\delta}
\def\a{\alpha}

\def\M{\mathcal M}

\def\R{\mathbb{R}}
\def\b{\beta}

\def\L{\Lambda}

\def\N{\mathbb{N}}

\def \d {\delta}
\def \L {\Lambda}

\def \o {\omega}

\newcommand{\indic}[1]{\mathbb{I}_{\{#1\}}}
\def\nn{\nonumber}

\begin{document}
\title{Continuous spin models on annealed generalized random graphs
}

\author{S. Dommers, C. K\"ulske, P. Schriever}

\address{S.\ Dommers\\ Fakult\"at f\"ur Mathematik,
Ruhr-University of Bochum, Postfach 102148,\,
44721, Bochum,
Germany}
\email {Sander.Dommers@ruhr-uni-bochum.de}

\address{C.\ K\"ulske\\ Fakult\"at f\"ur Mathematik,
Ruhr-University of Bochum, Postfach 102148,\,
44721, Bochum,
Germany}
\email {Christof.Kuelske@ruhr-uni-bochum.de}

\address{P.\ Schriever\\ Fakult\"at f\"ur Mathematik,
Ruhr-University of Bochum, Postfach 102148,\,
44721, Bochum,
Germany}
\email {Philipp.Schriever-d8j@ruhr-uni-bochum.de}

\begin{abstract} 
We study Gibbs distributions of spins taking values in a general compact Polish space, interacting via a pair potential along the edges of a generalized random graph with a given asymptotic weight distribution $P$, obtained by annealing over the random graph distribution. 

First we prove a variational formula for the corresponding annealed pressure and provide criteria for absence of phase transitions in the general case. 

We furthermore study classes of models with second order phase transitions which include rotation-invariant models on spheres and models on intervals, and classify their critical exponents. We find critical exponents which are modified relative to the corresponding mean-field values when $P$ becomes too heavy-tailed, in which case they move continuously with the tail-exponent of $P$. For large classes of models they are the same as for the Ising model treated in~\cite{DomGiaGibHofPri16}. On the other hand, we provide conditions under which the model is in a different universality class, and construct an explicit example of such a model on the interval. 
\end{abstract}

\maketitle

\begin{center}

\today

\end{center}


\section{Introduction}
Spin models on random graphs are interesting for many reasons. Random graphs can for example serve as models for complex network such as social, information, technological and biological networks~\cite{New03}. Spins can e.g.\ describe an opinion or an internal state of a person~\cite{MonSab10} or neurons in the brain~\cite{FraBalFosChi09}. The interaction between the behavior of spin models and the properties of the random graph is of special interest, see for example~\cite{DorGolMen08} for an overview of (often non-rigorous) results in the physics literature.

Also in the mathematics community there has been a large interest in such models recently. Especially the ferromagnetic Ising model on random graphs has attracted a lot of attention, both in equilibrium~\cite{DemMon10,DomGiaHof10,DomGiaHof14,DomGiaGibHofPri16,GiaGibHofPri15,GiaGibHofPri16} and out of equilibrium, i.e., the dynamics of this model~\cite{MosSly13,Dom15,DomHolJovNar16}. In this model, spins can only take two values and the spins tend to align. The Ising model with both ferromagnetic and antiferromagnetic couplings is studied in~\cite{GueTon04}.

When spins can take more than two values, the models often become more difficult to analyze, see, e.g.,~\cite{CojJaa16,ConDomGiaSta13,DemMonSlySun12,DemMonSun13} for some results where spins can take a finite number of values.

In this paper, we study a much more general setting where spins can take values in a general compact Polish space, in particular they can be continuous. Also the pair interaction potential between two neighboring spins is only assumed to be bounded. We study such models on {\em generalized random graphs} in the {\em annealed} setting.

In {\em generalized random graphs}, each edge (or bond) is present independently with 
probability which is (essentially) proportional to the product 
$w_i w_j$ where $w_i>0$ is a quenched weight variable associated to the 
site $i$. Hence $w_i$ can be seen as the affinity of the site $i$ to form connections. 
See~\cite{BolJanRio07} for an extensive analysis of this model. We thereby assume that 
the empirical weight distribution at finite $N$ converges in distribution to the distribution $P$
of a limiting variable $W$ taking values in the positive reals,  
such that the second moment converges, too. 

By {\em annealing} we average the exponential of the Hamiltonian of the spin model under investigation
over the random graph. Using the language of social networks, working in this annealed framework  
means that we are considering an equilibrium distribution in a regime where 
making (or losing) friends is happening on a faster time-scale than opinion-forming. There is also an interest in annealed spin models on random graphs in the context of quantum gravity, see e.g.\ \cite{NapTur16}.

A similar set-up has been studied for the particular case of the Ising model in~\cite{GiaGibHofPri16,DomGiaGibHofPri16}. In~\cite{GiaGibHofPri16}, the pressure is computed, showing that it is a function of the solution to some fixed point equation. Also a central limit theorem for the total spin is given.
In \cite{DomGiaGibHofPri16}, it was found that depending on tail-behavior of the limiting weight distribution $P$ 
the transition stays second order, but the critical exponents change from their 
mean-field values when the weights become too heavy-tailed to values which are computable in terms of the tail-exponent. This shows that the annealed Ising model on random graphs is in the same universality class as its quenched counterpart~\cite{DomGiaHof14}.

In the present work, we prove that the pressure also exists in our setting with more general state spaces and interactions and show that it satisfies a variational principle. We also investigate the critical behavior of a class of models with second order phase transition (at which order appears continuously). These models are rotator models which have $O(q)$-symmetry and models on the interval. We also investigate whether new exponents beyond the mean-field and heavy-tailed Ising case can be found, thus proving the existence of different universality classes.

\subsection{Results}
We now describe our results on the existence of the pressure in general and the critical behavior of certain models in more detail.
\subsubsection{Existence of annealed pressure, functional fixed point equation, absence of phase transition}

We derive a useful representation of the thermodynamic limit of the annealed pressure, if it exists, in Proposition \ref{pro:AnnPrRep} which 
is used frequently later. 
Here the annealed pressure is written as a mean-field expression involving the empirical distribution 
of the joint variables $(w_i,\s_i)$ given by weight-variable and spin-variable at the site $i$, and an effective pair potential $U$ 
in the space of joint variables. The problem is thereby reduced to a large-deviation analysis in which the empirical distribution 
of the weights is quenched, with asymptotic distribution $P$.  

Next, in Theorem \ref{thm:AnnPr} we state existence of the thermodynamic limit of the annealed pressure of the model, 
assuming convergence of empirical weight distribution (and its second moments) to $P$. 
We formulate our result for general compact Polish local state space $E$, 
but notably non-compact weight space $\R_{>0}$. 
The existence of the annealed pressure 
follows via the application of the G\"artner-Ellis theorem (Lemma~\ref{lem1}), where we need care to handle the unboundedness 
of the weights by a suitable truncation argument. 
Using Varadhan's lemma we show that the annealed 
pressure is realized as the maximum of a certain kind of functional on the set of probability measures which have their marginals on the weight space prescribed by $P$. 
As Theorem \ref{thm:MF} asserts, 
the stationarity condition, or (functional) mean field equation, then takes the form, with $\Phi$ denoting the original interaction potential,
\begin{equation}\begin{split}\label{MFIntro}
V=T_{\Phi,P,\alpha}(V),
\end{split}\end{equation}
where $V$ is now a function (an effective potential) on the single spin space. 
Hence we are left to study an (in general) infinite dimensional fixed point problem, depending 
on $\Phi,P$ and $\a$.

Such an infinite dimensional fixed point problem with its transitions between uniqueness regions and non-uniqueness regions which can be connected by various 
bifurcation scenarios is a source of great richness in general. 
We provide a general uniqueness criterion in Theorem~\ref{thm:unique} in terms 
of sized biased expectation of weights $W$ (describing the effective degree of the network) 
times the second order variation of $e^{\Phi}$ (describing the effective interaction strength). 
Checked against the known case of the Ising model the condition 
gives a bound on the critical 
temperature which is of the correct order of magnitude, but non-optimal prefactor. 

It is worthwhile to contrast this constructive result which allows us to describe the system obtained 
by annealing over graphs at fixed weights as a mean-field Gibbs model with joint potential $U$ 
with a ``negative'' result 
in a related but slightly different situation: 
In the so-called Morita-approach to disordered systems 
of theoretical physics~\cite{Mor64, Kuh94, EntKulMae00} one tries to interpret a quenched model as a formal Gibbsian model 
with a new Hamiltonian depending both on disorder variables and spin variables. 
The original motivation to do so stems from non-rigorous renormalization group theory 
with an aim to determine critical exponents of the random system.  
It has been shown in this context that a Gibbsian description of such joint measures 
is in general not possible, when one asks for a well-behaved Hamiltonian. This is made clear in examples 
based on the random field Ising model described in detail in~\cite{KulLenRed04,Kul03}.
For some general background on the occurrence of non-Gibbsian measures, 
Gibbs-non-Gibbs transitions, and constructive use of the preservation of the 
Gibbs property,  see~\cite{EntFerSok93,EntFerHolRed10,KulLen07,HolRedZui15,JahKul16,RoeRus14}.

\subsubsection{Low-rank models, second order phase transitions, critical exponents}

To work right at the critical region we narrow down our models. 
While the previous results were completely general, in the following part of the paper we are only 
interested in models with second order phase transitions.  
We are driven mainly by two questions: Can we incorporate $O(q)$-invariant 
rotator models in our analysis? Do they lead to the same critical exponents as the Ising model?  
That is: do they have 
 the same weight-tail dependent change from the standard mean-field exponents 
 as the Ising model? 
Further,  can we find models which have different behavior than the Ising model at all? 
 
We answer both questions within a specific class of models: 
In Section~\ref{sec-criticalbehavior5.1}, we specialize to models
on the single-spin space $[-1,1]$ which interact 
via a low-rank kernel (of rank $2$) of the form $e^{\Phi(\s,\s')}=c+ \theta g(\s)g(\s')$, 
where $c>0$ is fixed,  $g$ is an odd function on $[-1,1]$ and 
$\theta$ is a coupling constant playing the role of an inverse temperature.

We choose for $\a_0$ a symmetric 
probability measure on $[-1,1]$ and define the action of a real valued  external 
magnetic field $h$ to the model in terms of tilted measures with Radon-Nikodym derivative 
$\frac{d\a}{d\a_0}(\s)=e^{h \s}/z$ where $z$ is a normalizing constant. 
When we later speak 
about critical exponents, they will be formulated w.r.t.\ the inverse temperature variable $\theta$, 
and the magnetic field variable $h$. The assumption of negativity of the tilted third cumulant of the single-spin function 
$g$ w.r.t.\ $\a_0$, for positive tilts,  will ensure that the phase transition is of second order. 
Our analysis also applies to  models for rotators $\s\in S^q$ taking values in the $q$-dimensional sphere 
with the $O(q)$-invariant interaction $e^{\Phi(\s,\s')}=c+ \theta\langle \s,\s'\rangle$, see Section~\ref{sec-criticalbehavior5.4}. 
They are believed to be in the same universality class as the ones with the standard interaction 
$e^{\Phi(\s,\s')}=e^{\theta\langle \s,\s'\rangle}$~\cite{DorGolMen08}. This is a  statement whose full mathematical justification 
would however need an investigation of the corresponding infinite-dimensional fixed point problem, and would be 
an interesting analytical problem in itself for future study. 

The idea to assume the exponential of a potential 
to be of a nice form consisting of two ``simple terms'', has appeared previously and led 
to fruitful results.  
Decompositions in such a spirit  into a sum of two Gaussians, 
have e.g.\ been used in~\cite{BisKot07} for pair potentials and in~\cite{KulRed06} for single-site potentials. 
See also~\cite{JahKulBot14} for a model on the tree.

Our main result on low-rank kernel models 
is Theorem \ref{thm-critexp} which gives the values of temperature critical 
exponent $\boldsymbol{\beta}$, and magnetic field critical exponent $\boldsymbol{\delta}$, in terms of distribution of $g$ w.r.t.\  
$\a_0$, and tail-behavior of $P$. 
Answering our first question, we find that the $O(q)$-rotator models are covered by our analysis using 
beta distributions, and do indeed show 
the very same mean field/modified mean field scenario as the Ising model does.  
In the light of Theorem \ref{thm-critexp} this should even be seen as generic.  

Answering our second question, 
it is possible though to come up with excess kurtosis-zero models which do show 
a mean field/modified mean field scenario of {\em different than Ising type}, for which we construct 
an explicit example in  Proposition \ref{step}.

\section{Annealing over the random graphs}

\subsection{Model definitions}
In the following we study spin models on sequences of {\em generalized random graphs}. 
These graphs consist of a vertex set $\left[ N \right] := \{ 1,..., N \}$, and a random set of edges $E_N$, where 
an edge between vertices $i,j \in \left[ N \right]$ is denoted by $( i,j )$.  
Each vertex $i \in \left[ N \right]$ receives a weight variable $w_i$ which takes values in the space of the non-negative reals $\R_{>0}$. Given these weights, an edge $( i,j )$ will be present with a certain probability $p_{i,j}$, independently of all the other edges. These edge probabilities $p_{i,j}$ are moderated by the weights assigned to the vertices. 

\begin{defn}\label{defn:GRG}
Denote by $I_{ij}$ independent Bernoulli variables indicating that an edge between vertices $i$ and $j$ is present with $p_{i,j} = Q^w (I_{i,j} = 1)$. Then the {\em generalized random graph} with vertex set $\left[ N \right]$, denoted by $\text{GRG}_{N,w}$, is defined by
\begin{equation}\begin{split}\label{Qw}
Q^w(I_{i,j}=1)=p_{ij}=\frac{w_i w_j}{l_N + w_i w_j} ,\cr
\end{split}\end{equation}
where $l_N=\sum_{i=1}^N w_i$ is the total weight sum. We denote by $Q_N^w$ the law of $\text{GRG}_{N,w}$. 
\end{defn}
 
As we are interested in spin models on sequences of generalized random graphs as $N \to \infty$, we need to assume that the vertex weights are relatively nice behaved. 
Let $V_N$ denote a uniformly chosen vertex from $\{1,\ldots,N\}$ and $W_N=w_{V_N}$. We assume that the sequence of weights $(W_N)_{N \in \mathbb{N}}$ satisfies the following condition which defines an asymptotic weight random variable $W$:

\begin{con}\label{con:weight}
There exists a random variable $W$ such that, as $N \to \infty$,
\begin{enumerate}[(i)]
\item
\begin{equation}
W_N \stackrel{\mathcal{D}}{\longrightarrow} W, 
\end{equation}
\item
\begin{equation} \label{eq-2ndmomentconv}
\E[W_N^2] = \frac{1}{N} \sum_{i\in\left[ N \right]} w_i^2 \longrightarrow  \E[W^2] <\infty,
\end{equation}
\end{enumerate}
where $\stackrel{\mathcal{D}}{\longrightarrow}$ denotes convergence in distribution. 
\end{con}

Note that the assumed $L^2$-convergence of Condition \ref{con:weight} readily implies convergence in $L^1$, i.e.
\begin{equation}
\E [W_N] = \frac{1}{N} \sum_{i \in [N]} w_i \longrightarrow \E[W] < \infty,
\end{equation}
which is due to the uniform integrability of the weight sequence $(W_N)_{N\in\mathbb{N}}$.

Let a random field $\sigma = (\sigma_i)_{i \in \left[ N \right]}$ taking values on a local polish state space $E$ be given. Furthermore, let $\Phi: E \times E \to \R$ be a bounded nearest-neighbor interaction potential. 
In the following we study the {\em annealed spin measure} given by annealing over the edge set $E_N$ which is given by 
\begin{equation}\begin{split}
P^w_N(d\s_{[N]})=\frac{Q_N^w\Bigl( e^{\sum_{(i,j)\in E_N}\Phi(\s_i,\s_j)}\prod_{i\in [N]}\a(d\s_i)\Bigr)}
{Q_N^w \a^{N}\Bigl( e^{\sum_{(i,j)\in E_N}\Phi(\cdot,\cdot)}\Bigr)
},
 \cr
\end{split}\end{equation}
where $\a$ is a probability measure on $E$ which may incorporate also a magnetic field term. 

\begin{defn}
For a given finite volume $N \in \mathbb{N}$ we define the {\em annealed pressure} to be 
\begin{equation}\begin{split}\label{AnnPr}
\psi_N^w(\Phi,\a):=\frac{1}{N}\log Q_N^w \a^{N}\Bigl( e^{\sum_{(i,j)\in E_N}\Phi(\cdot,\cdot)}\Bigr).
 \cr
\end{split}\end{equation}
If the thermodynamic limit of the annealed pressure is well-defined, i.e. the limit $\psi(\Phi,P,\a) := \lim_{N\to\infty} \psi_N^w(\Phi,\a)$ exists and is finite, then we call $\psi(\Phi,P,\a)$ the {\em pressure} of the system. 
\end{defn}

Clearly the annealed pressure will generally depend on the fixed realization of weights $w$ (through their empirical distribution), the pair interaction 
potential $\Phi$ in spin space, and the a priori measure $\a$ on the single spin space. 

\subsection{Representation of the pressure via exponential integrals}
We can represent the pressure as an exponential integral of empirical distributions of the spins and weights $L^{(\s,w)}_N=\frac{1}{N}\sum_{i\in [N]}\d_{(\s_i,w_i)}$:
\begin{pro}\label{pro:AnnPrRep}
Let a sequence of generalized random graphs $(\text{GRG}_{N,w})_{N \in \mathbb{N}}$ be given which satisfies Condition \ref{con:weight}. If the thermodynamic limit of the annealed pressure exists, the pressure $\psi(\Phi,P,\a)$ is given by
\begin{equation}\begin{split}
\psi(\Phi,P,\a) = \lim_{N \to \infty} \frac{1}{N} \log \a^N \left(  \exp \left( N  L^{(\s,w)}_N\otimes L^{(\s,w)}_N (U)\right) \right)
\end{split}\end{equation}
where $U: E^2 \times E'^2 \to \R$ is given by
\begin{equation}
U(\s, \s', w, w') = \frac{ww'}{2 \E[W]} e^{\Phi(\s,\s')}.
\end{equation}
\end{pro}

\begin{proof}
Carrying out the expectation over random graphs we have 
\begin{equation}\begin{split}\label{ExRanGr}
Q_N^w \Bigl( e^{\sum_{(i,j)\in E_N}  \Phi(\s_i,\s_j)}\Bigr)
&= \exp \left( \frac{1}{2} \sum_{\substack{i,j\in [N] \\ i\neq j}} \log( 1+ p_{i,j}(e^{ \Phi(\s_i,\s_j)}-1)) \right)
 \cr
 &= \exp \left(\frac{1}{2} \sum_{i,j\in [N]} p_{i,j}(e^{ \Phi(\s_i,\s_j)}-1) \right) \cr
&  \times \exp \left( \frac{1}{2} \sum_{i,j\in [N]} 
 \log( 1+ p_{i,j}(e^{ \Phi(\s_i,\s_j)}-1))-p_{i,j}(e^{\Phi(\s_i,\s_j)}-1) \right) \cr
 &  \times \exp \left( -\frac{1}{2} \sum_{i\in[N]} \log(1+p_{i,i}(e^{ \Phi(\s_i,\s_i)}-1)) \right) \cr
& = \exp \left( \frac{1}{2} \sum_{i,j\in [N]} 
\frac{w_i w_j}{l_N + w_i w_j} e^{\Phi(\s_i,\s_j)} \right)
\times \exp \left(- \frac{1}{2}\sum_{i,j\in [N]} p_{i,j} \right) \cr
&\times \exp(R_N).
\end{split}\end{equation}

Denote by $\Phi_+$ the sup and by $\Phi_-$ the inf over pairs in $E$ of the potential. Using a Taylor expansion
\begin{equation}
\log( 1+ p_{i,j}(e^{ \Phi(\s_i,\s_j)}-1))-p_{i,j}(e^{ \Phi(\s_i,\s_j)}-1) = -\frac{1}{(1+\xi)^2} \frac12 p_{i,j}^2 (e^{ \Phi(\s_i,\s_j)}-1)^2,
\end{equation}
for some $\xi$ between  $0$ and $p_{i,j}(e^{ \Phi(\s_i,\s_j)}-1)$. If $\Phi(\s_i,\s_j)\geq 0$, we can bound this from below by $-\frac12 p_{i,j}^2 (e^{ \Phi_+}-1)^2$, and if $\Phi(\s_i,\s_j)< 0$, we can bound it from below by $-\frac{1}{e^{\Phi_-}}\frac12 p_{i,j}^2 (e^{ \Phi_-}-1)^2$. Hence, also using that $0\leq  \log(1+x) \leq x$ for all $x>-1$,
\begin{equation}
0\geq R_N \geq  -\frac{1}{4}\frac{(e^{ \Phi_+}-1)^2+(e^{ \Phi_-}-1)^2}{e^{\Phi_-}\wedge 1} \sum_{i,j\in [N]} \left(\frac{w_i w_j}{l_N}\right)^2- \frac{1}{2} (e^{ \Phi_+}-1) \sum_{i\in[N]} \frac{w_i^2}{l_N},
\end{equation}
which is seen to be of smaller order than $N$ because of~\eqref{eq-2ndmomentconv}.

Write 
\begin{equation}\begin{split}
& \frac{1}{2} \sum_{i,j\in [N]} 
\frac{w_i w_j}{l_N + w_i w_j} e^{ \Phi(\s_i,\s_j)}
= \frac{1}{2} \sum_{i,j\in [N]} 
\frac{w_i w_j}{l_N} e^{ \Phi(\s_i,\s_j)} + R_N',
\end{split}\end{equation}
where 
\begin{equation}\begin{split}
&0\geq R_N'
\geq \frac{1}{2} \sum_{i,j\in [N]} w_i w_j
\left( \frac{1}{l_N + w_i w_j} -\frac{1}{l_N } \right) e^{ \Phi_+}
\geq  -\frac{e^{\Phi_+}}{l^2_N} \sum_{i,j\in [N]} w^2_i w^2_j,
\end{split}\end{equation}
which again is of smaller order than $N$ because of~\eqref{eq-2ndmomentconv}. 

We can write
\begin{equation}
\frac{1}{2} \sum_{i,j\in [N]} \frac{w_i w_j}{l_N} e^{ \Phi(\s_i,\s_j)}  =  \frac{1}{2N \E[W]}\sum_{i,j\in [N]} w_i w_j e^{ \Phi(\s_i,\s_j)}+R_N'',
\end{equation}
where
\begin{equation}\begin{split}
R_N'' &= \frac{1}{2 N} \left(\frac{1}{\E[W_N]}-\frac{1}{\E[W]}\right) \sum_{i,j\in [N]} w_i w_j e^{\Phi(\s_i,\s_j)}\\
 &= \left(\frac{\E[W]-\E[W_N]}{\E[W_N]\E[W]}\right)\mathcal{O}(N) = o(N),
\end{split}\end{equation}
because of our assumptions on the weights.

In terms of the empirical distribution of joint spins $L^{(\s,w)}_N
=\frac{1}{N}\sum_{i\in [N]}\d_{(\s_i,w_i)}$ we have the exact identity 

\begin{equation}\begin{split}
& \frac{1}{2 N\E[W]}\sum_{i,j\in [N]} 
w_i w_j e^{\Phi(\s_i,\s_j)}=N  L^{(\s,w)}_N\otimes L^{(\s,w)}_N(U).
\end{split}\end{equation}
It has a pleasant product form w.r.t.\ the decomposition over spin-part and weight-part. 

Summarizing this rewriting has given us  
\begin{equation}\begin{split}
Q_N^w \Bigl( e^{\sum_{(i,j)\in E_N}\Phi(\s_i,\s_j)}\Bigr)
&= \exp \left( N  L^{(\s,w)}_N\otimes L^{(\s,w)}_N (U)\right)\\
& \qquad \times \exp \left(- \frac{1}{2} \sum_{i,j\in [N]} p_{i,j}+R_N+R_N'+R_N'' \right).
\end{split}\end{equation}
Only the first exponential is interesting. 
\end{proof}

\section{LDP via abstract G\"artner-Ellis Theorem}

\subsection{LDP for random measures}

We have seen that the computation of the pressure of the annealed model leads to the evaluation of exponential integrals of the pair empirical distributions $L_N^{(\sigma, w)}$, i.e.,
\begin{equation}\label{expInt}
\psi(\Phi,P,\a) = \lim_{N \to \infty} \frac{1}{N} \log \a^N \left( \exp \left( N  L_N^{(\sigma,w)} \otimes  L_N^{(\sigma,w)} (U) \right) \right). 
\end{equation}
In the following we show that the random pair empirical distributions $(L_N^{(\sigma, w)})_{N \in \mathbb{N}}$ satisfy a large deviation principle (LDP), so that it is possible to compute \eqref{expInt} by using Varadhan's lemma \cite[Theorem 4.3.1]{DemZei09}. 

Look at random measures on a polish space $\Gamma$, which can be interpreted as elements of the dual space $\mathcal{C}_b(\Gamma)^*$. Let $\mathcal{Y} := \{ E_f: \mathcal{C}_b(\Gamma)^* \to \mathbb{R} \mid E_f(\phi) = \phi(f) , f \in \mathcal{C}_b(\Gamma) \}$ be the subset of the evaluation functionals of $\mathcal{C}_b(\Gamma)^*$. 
We endow the dual space $\mathcal{C}_b(\Gamma)^*$ with the coarsest topology s.t. all elements of $\mathcal{Y}$ remain continuous, which is also called the \textit{weak$^*$ topology}. Clearly the set $\mathcal{Y}$ is \textit{separating} in the sense that for any $\phi \in \mathcal{C}_b(\Gamma) \setminus \{0\}$ there exists a $E_f \in \mathcal{Y}$, s.t. $E_f(\phi) \neq 0$. Hence the dual space $\mathcal{C}_b(\Gamma)^*$ is a locally convex Hausdorff topological vector space w.r.t.\ to the weak$^*$ topology and furthermore the double dual space is then given by $\mathcal{Y}$, i.e. $\mathcal{C}_b(\Gamma)^{**} = \mathcal{Y}$~\cite[Theorem B.8]{DemZei09}.

\begin{lem}\label{lem1}
Let $\Gamma$ be a Polish space and $\left( \mu_n \right)_{n\in\mathbb{N}}$ an exponentially tight sequence of $\mathcal{M}_1(\Gamma)$-valued random variables. Let $(\gamma_n)_{n\in\mathbb{N}}$ be a sequence of positive numbers with $\gamma_n \to \infty$ for $n \to \infty$. 
Assume that the limit
\begin{equation}
\Lambda(f) = \lim_{n\to\infty} \frac{1}{\gamma_n} \log \E \left[ e^{\gamma_n \langle \mu_n, f \rangle} \right],
\end{equation}
exists for every $f \in \mathcal{C}_b(\Gamma)$ and is finite. If the map $\Lambda: \mathcal{C}_b(\Gamma) \to \R $ is G\^{a}teaux differentiable and continuous at $0$ in the sense that $\lim_{n\to \infty} \Lambda(f_n) =0$ for every sequence of non-negative $f_n \in \mathcal{C}_b(\Gamma)$ with $f_n \downarrow 0$ pointwise, then the sequence of random measures $(\mu_n)_{n\in\mathbb{N}}$ satisfies the LDP with rate $\gamma_n$ and good rate function
\begin{equation}
I(\mu) = \sup_{f \in \mathcal{C}_b(\Gamma)} \left[ \langle \mu, f \rangle - \Lambda(f) \right], \quad \mu \in \mathcal{M}_1(\Gamma).
\end{equation}
\end{lem}

This lemma is a consequence of the G\"artner-Ellis theorem~\cite[Corollary~4.6.14]{DemZei09}, which already gives us that the sequence $(\mu_n)_{n\in\mathbb{N}}$ satisfies an LDP on $\mathcal{C}_b(\Gamma)^*$ with the good rate function being the Fenchel-Legendre transform 
\begin{equation}
I(\mu) = \sup_{f\in\mathcal{C}_b(\Gamma)} \left[ \langle \mu,f \rangle - \Lambda(f) \right].
\end{equation}
The LDP can then be restricted to the subset $\mathcal{M}_1(\Gamma)$ of probability measures by showing that $I(\mu) = \infty$ for every $\mu \in \mathcal{C}_b(\Gamma)^* \setminus \mathcal{M}_1(\Gamma)$, which can be seen by using the Daniell-Stone theorem \cite[Theorem 7.8.1.]{Bo07}.

This property is needed to show the upper bound of the LDP: A closed set $F$ in $\mathcal{M}_1(\Gamma)$ is the intersection of a closed set $\tilde F \subset \mathcal{C}_b(\Gamma)^*$ with $\mathcal{M}_1(\Gamma)$. Hence it is necessary that $\inf_F I = \inf_{\tilde F} I$. \\

\subsection{An LDP for bounded weights}

\begin{thm}\label{thm:LDPtrunc}
Assume that the weights $(w_i)_{i \in \mathbb{N}}$ satisfy Condition~\ref{con:weight} and furthermore that they only take values in the compact interval $E':=[0, R]$ for some $R>0$. Let the spin variables $\sigma_i$  be distributed according to an a priori measure $\a \in \mathcal{M}_1(E)$ for every $i \in \N$.
Then the pair empirical distributions $\left( L_N^{(\sigma, w)} \right)_{N \in \mathbb{N}}$ satisfy an LDP on $\mathcal{M}_1(E \times E')$ with good rate function 
\begin{equation}\begin{split}\label{ga}
I(\nu)= \begin{cases} \int P(dw ) S(\nu^w |\a) &\mbox{if } \nu(dw)=P(dw) \\ 
\infty & \mbox{else, } \end{cases}
\end{split}\end{equation}
where $S(\nu |\a)$ is the relative entropy, i.e.,
\begin{equation}
S(\nu |\a) = \int \nu(d\s) \log \frac{d\nu}{d\a}(\s).
\end{equation}
\end{thm}

\begin{proof}
We fix a weight sequence $w_i \in E'$, $i \in \N$. 
Look at the random element $L_n=\frac{1}{n}\sum_{i=1}\d_{w_i,\s_i}\in \mathcal{M}_1(E\times E')$ 
with deterministic values of $w_i$ and $\s_i$ i.i.d. drawn from $\a$. 
Recall that the product space $E \times E'$ is compact by assumption. Hence the empirical distributions $L_n$ are exponentially tight. 

The logarithmic moment generating function $\L_n(\cdot)$ at size $n$ is given by
\begin{equation}\begin{split}
\L_n(f) &:= \log \E \left[ e^{ \langle L_n, f \rangle} \right] = \log \a^n \left(\exp \left[ \frac{1}{n}\sum_{i=1}^n f(w_i,\s_i) \right] \right) \cr
&= \sum_{i=1}^n \log \a \left(\exp \left[ \frac{1}{n}f(w_i,\s) \right] \right).
\end{split}\end{equation}

We need to look at the limit of
\begin{equation}\begin{split}
&\frac{1}{n}\L_n(n f)= \frac{1}{n}\sum_{i=1}^n \log \a (\exp \left[ f(w_i,\s) \right] ).
\end{split}\end{equation}

If the empirical distribution on the $w_i$'s converges in the weak topology towards a measure $P(dw)$  
then also we have the convergence 
\begin{equation}\begin{split}\label{eq:lambda}
&\lim_{n \to \infty} \frac{1}{n}\L_n(n f)= \int P(dw ) \log \a (\exp \left[ f(w,\s) \right] )=: \L(f).
\end{split}\end{equation}

By the dominated convergence theorem follows that the map $\L:  \mathcal{C}_b(E \times E') \to \R$ is continuous at $0$. To see the G\^{a}teaux differentiability let $f,g \in \mathcal{C}_b(\Gamma)$. Then
\begin{equation}\begin{split}
\frac{1}{t} \left[ \L (f+tg) - \L(f) \right] &=  \frac{1}{t} \left[ \int P(dw) \log \left( \a \left( e^{f+tg} \right)\right) - \log \left( \a\left( e^f \right) \right) \right] \\
&=  \frac{1}{t} \E \left[ \log \left( \a \left( e^f (1+tg + R_t) \right) \right) - \log \left( \a \left( e^f\right) \right) \right],
\end{split}\end{equation} 
where $R_t$ is the quadratic remainder in the Taylor expansion of $e^{tg}$. This gives us
\begin{equation}\begin{split}
\frac{1}{t} \left[ \L (f+tg) - \L(f) \right] 
&=  \frac{1}{t} \E \left[ \log \left( \frac{\a \left( e^f (1+tg + R_t) \right)}{\a \left( e^f\right)} \right) \right] \\
&= \frac{1}{t} \E \left[ \log \left( 1 + t \frac{\a (e^f g)}{\a (e^f)} + \frac{\a(e^f R_t)}{\a(e^f)} \right) \right].
\end{split}\end{equation} 
Expanding the logarithm shows that 
\begin{equation}
\lim_{t \to0} \frac{1}{t} \left[ \L (f+tg) - \L(f) \right] = \E \left[ \frac{\a(e^fg)}{\a(e^f)} \right],
\end{equation}
which shows that $\L$ is indeed G\^{a}teaux differentiable. Hence we can apply Lemma \ref{lem1} and obtain that the empirical distributions satisfy the LDP with rate function $I(\cdot)$ given by the Fenchel-Legendre transform of $\L(\cdot)$, i.e.,
\begin{equation}
I(\nu) = \sup_{f \in \mathcal{C}_b(E \times E')} \left[ \langle \nu, f \rangle - \L(f) \right].
\end{equation}

In the following we prove that 
\begin{equation}\begin{split}\label{eq:rate}
I(\nu)= \begin{cases} \int P(dw ) S(\nu^w |\a) &\mbox{if } \nu(dw)=P(dw) \\ 
\infty & \mbox{else. } \end{cases}
\end{split}\end{equation}
As the empirical distributions $L_n$ satisfy an LDP with good rate function $I$ it follows from Varadhan's lemma~\cite[Theorem~4.5.10~a)]{DemZei09}, that $\L(\cdot)$ is itself the Fenchel-Legendre transform of $I$, i.e.
\begin{equation}\label{eq:FLT}
\L(f) = \sup_{\nu \in \mathcal{C}_b(E \times E')^*} \left[ \langle \nu , f \rangle - I(\nu) \right].
\end{equation}

We already noted that $I(\nu) = \infty$ for all $\nu \in \mathcal{C}_b(E \times E')^* \setminus \mathcal{M}_1(E \times E')$. The same is true for probability measures $\nu$ whose marginal distribution on the weight space $E'$ is not equal to $P$, i.e. $\nu(dw) \neq P(dw)$:

\begin{equation}\begin{split}
I(\nu) &= \sup_{f \in \mathcal{C}_b(E\times E')} \left[ \nu(f) - \L(f) \right] 
\geq \sup_{g \in \mathcal{C}_b(E')} \left[ \nu(g) - P(g) \right],
\end{split}\end{equation}
where the supremum in the second line is taken over all functions $g \in \mathcal{C}_b(E')$ which are constant on the local state space. 
As $\nu(dw) \neq P(dw)$ there exists a function $\tilde g \in \mathcal{C}_b(E')$ s.t. $\nu(\tilde g) \neq P (\tilde g ) $ and therefore
\begin{equation}
I(\nu) \geq \sup_{\l \in \R} \l(\nu (\tilde g) - P(\tilde g)) = \infty.
\end{equation}

Via duality of the Fenchel-Legendre transform we can prove \eqref{eq:rate} by checking that representation \eqref{eq:lambda} is recovered when plugging \eqref{eq:rate} into \eqref{eq:FLT}. Doing so we arrive at
\begin{equation}\begin{split}
 \sup_{\nu \in \mathcal{C}_b(E \times E')^*} \left[ \nu(f)- I(\nu) \right] 
&=\sup_{\nu \in \mathcal{M}_1(E\times E'): \nu(dw)=P(dw)} \left[ \nu(f)- \int P(dw ) S(\nu^w |\a) \right] \cr
&=\sup_{\nu \in \mathcal{M}_1(E\times E'): \nu(dw)=P(dw)}  \left[ \int P(dw )( \nu^w(f)- S(\nu^w |\a)) \right]. \cr
\end{split}\end{equation}
But the sup now breaks down into a sup over the marginals $\nu^w \in \mathcal{M}_1(E)$ for which 
we know that 
\begin{equation}\begin{split}
&\sup_{\rho \in \mathcal{M}_1(E)}( \rho(f)- S(\rho |\a))=\log \a (\exp [ f(w,\s) ] ). \cr
\end{split}\end{equation}
This finishes the proof. 
\end{proof}

\subsection{LDP for weights with strongly finite mean}
We next prove an LDP for the mean-field expression derived in Proposition~\ref{pro:AnnPrRep}. Instead of assuming that the second moment is finite, we prove this under the weaker condition that the weight sequence has {\em strongly finite mean}, that is, we assume that, for all $\varepsilon>0$,
\begin{equation}\begin{split}\label{ggg}
&\sup_{N \in \N}  L_N (w^{1+\varepsilon})<\infty.
\end{split}\end{equation}

\begin{thm}\label{thm:AnnPr}
Assume that the weights $(w_i)_{i \in \mathbb{N}}$ satisfy Condition~\ref{con:weight}(i) and have strongly finite mean. Let $\Phi: E \times E \to \R$ be a bounded pair interaction potential and $\a \in \mathcal{M}_1(E)$ an a priori measure on the state space. Then,
\begin{equation}\begin{split}
\lim_{N \to \infty} \frac{1}{N} &\log \a^N \left(  \exp \left( N  L^{(\s,w)}_N\otimes L^{(\s,w)}_N (U)\right) \right) \\
&= \sup_{\nu \in \mathcal{M}_1(E\times E'): \nu(dw) = P(dw)} \left( \nu \otimes \nu (U) - \int P(dw) S(\nu^w\mid\a) \right)< \infty,
\end{split}\end{equation}
where
\begin{equation}
U(\sigma, \sigma', w, w') = \frac{ww'}{2 \E[W]} e^{ \Phi(\sigma, \sigma')}.
\end{equation}

\end{thm}

\begin{proof}
Write convex combinations of the joint empirical measure 
in terms of the truncated empirical measure 
\begin{equation}\begin{split}
 L_N=\frac{1}{N}\sum_{i\in [N]} \d_{w_i,\s_i}
 &=c_N(R) L^{\leq R}_N +  (1-c_N(R)) L^{> R}_N,
 \end{split}\end{equation}
with the empirical distribution of the truncated variables 
\begin{equation}\begin{split}
 L^{\leq R}_N &= \frac{1}{N^{\leq R}}\sum_{i\in [N]: w_i\leq R} \d_{w_i,\s_i}, \quad L_N^{>R} = \frac{1}{N^{>R}} \sum_{i\in [N]: w_i >R} \delta_{w_i, \sigma_i}, \cr
 N^{\leq R} &=\#\{i\in [N]: w_i\leq R\}, \quad N^{> R} =\#\{i\in [N]: w_i > R\},\cr
 \end{split}\end{equation}
and $c_N(R) = N^{\leq R} / N$. 

Let us assume that the truncation threshold $R$ does not 
get mass w.r.t.\ the limiting distribution $P$. 
By the assumed weak 
convergence of the empirical distribution of the weights 
we then have the convergence against the conditional distribution 
$\frac{1}{N^{\leq R}}\sum_{i\in [N]: w_i\leq R} \d_{w_i}\rightarrow 
P(\cdot | W\leq R)$. 
 
We have proved that the truncated joint distribution 
$L^{\leq R}_N$ obeys an LDP with rate $N^{\leq R}$
 and rate function 

\begin{equation}\begin{split}
I^R(\nu)= \begin{cases} \int P(dw | W\leq R) S(\nu^w |\a) &\mbox{if } \nu(dw)=P(dw| W\leq R) \\ 
\infty & \mbox{else. } \end{cases}
\end{split}\end{equation}

Now we use that from the weak convergence of the weight distribution 
we have that the tails converge 
$\lim_{N \to \infty} c_N(R)=P(W\leq R)$. 
Using this truncation and taking first the limit $N\uparrow \infty$ and then the limit $R\uparrow \infty$ yields the desired result: 

To be more detailed, we introduce the above decomposition into 
\begin{equation}\begin{split}
&N  L^{(\s,w)}_N\otimes L^{(\s,w)}_N(U) 
=N^{\leq R} L^{\leq R}_N \otimes L^{\leq R}_N(U)
+ R_N,
\end{split}\end{equation}
where 
\begin{equation}\begin{split}
R_N &= N^{\leq R}(c_N(R) -1 ) L^{\leq R}_N\otimes L^{\leq R}_N(U)
 + 2 N^{\leq R}(1-c_N(R)) L^{\leq R}_N\otimes L^{> R}_N(U) \cr
& \qquad +  N (1-c_N(R))^2 L^{> R}_N\otimes L^{> R}_N(U).
\end{split}\end{equation}

We have by Varadhan's lemma \cite[Theorem 4.3.1]{DemZei09} and the LDP for truncated 
weight space (Theorem \ref{thm:LDPtrunc}) that 
\begin{equation}\begin{split}\label{MaxTrunc}
\lim_{N\to \infty}& \frac{1}{N^{\leq R}}\log \a^N (\exp 
N^{\leq R} L^{\leq R}_N\otimes L^{\leq R}_N(U)
)\cr
=&
\sup_{\bar\nu\in \M_1(E\times \R_+): \bar \nu(d w)=P(dw|W\leq R)}
\Biggl( \bar \nu \otimes \bar \nu (U)
- \int P (dw | W \leq R) S(\nu^{w}|\a )\Biggr)\cr
=&:F(\Phi,P,\a; R).
\end{split}\end{equation}

Note that $|U (\s, \s', w, w')| \leq C w w' $ where $C$ is a spin-independent constant
which depends only on the sup-norm of the potential acting 
on the pairs of spins $\Phi$, i.e. $C = e^{\frac{1}{2} ||\Phi||_\infty}/ \E W $.
Using this we find 
\begin{equation}\begin{split}
\frac{|R_N|}{C} & \leq N^{\leq R} |c_N(R)-1 | (L^{\leq R}_N(w))^2
+ 2 N^{\leq R}(1-c_N(R)) L^{\leq R}_N(w)L^{>R}_N(w) \cr
& \qquad + N (1-c_N(R))^2 L^{>R}_N(w)^2.
\end{split}\end{equation}
Recall our hypothesis~\eqref{ggg}, so that we have that the conditional empirical probabilities converge 
towards their corresponding values in the limiting distribution $W$, for instance 
\begin{equation}\begin{split}\label{g}
&\lim_{N\to\infty} L^{>R}_N(w)=\E(W|W > R).
\end{split}\end{equation}

\begin{equation}\begin{split}
\frac{|R_N|}{C N} 
\leq& \frac{N^{>R}}{N^{\leq R}} \left( \frac{1}{N}\sum_{i=1}^N w_i 1_{w_i\leq R} \right)^2\cr
&+ 2 \left( \frac{1}{N}\sum_{i=1}^N w_i 1_{w_i\leq R} \right) \left( \frac{1}{N}\sum_{i=1}^N w_i 1_{w_i
> R} \right)+ \left( \frac{1}{N}\sum_{i=1}^N w_i 1_{w_i> R} \right)^2.
\end{split}\end{equation}

Take first the $N$-limit to get for the r.h.s. (if the point
 $R$ gets no mass w.r.t.\ the distribution of $W$) that

\begin{equation}\begin{split}
\limsup_{N\to\infty}\frac{|R_N|}{C N} 
&\leq \frac{P (W > R)}{P (W \leq R)} (\E(W 1_{W\leq R}))^2 \cr
& \qquad + 2 \E(W 1_{W\leq R})\E(W 1_{W> R}) 
+  \E(W 1_{W> R})^2.
\end{split}\end{equation}
Now take the limit $R\uparrow \infty$ to get zero for the r.h.s. 

It remains to see that 
\begin{equation}\begin{split}
&\lim_{R \to \infty}F(\Phi,P,\alpha; R)=F(\Phi,P,\alpha):=\sup_{\bar \nu(d w)=P(dw)}
\Biggl( \bar \nu \otimes \bar \nu (U)
- \int P (dw) S(\bar \nu^{w}|\a )\Biggr).
\end{split}\end{equation}

We have that there exists a maximizer $\bar \nu^*$ such that 
\begin{equation}\begin{split}\
F(\Phi, P, \alpha)
&=\bar \nu^* \otimes \bar \nu^* (U)
- \int P (dw) S(\bar \nu^{w,*}|\a ).
\end{split}\end{equation}
This is the case since the set of measures $M_P:=\{ \nu \in \mathcal{M}_1(E \times E') : \nu(dw) = P(dw) \}$ is weakly closed and tight and therefore weakly compact by Prohorov's theorem \cite[Theorem 16.3]{K02}. The tightness of $M_P$ is easily seen as the state space $E$ is itself compact and every measure $\nu \in M_P$ has $P(dw)$ as its marginal on $E'$. As the single measure $P$ is tight, there exists for every $\epsilon >0$ a compact set $K \subset E'$ s.t. $P(K^c) < \epsilon$ and hence $\nu((E \times K)^c) < \epsilon$ for every $\nu \in M_P$. That $M_P$ is weakly closed follows from a simple application of the continuous mapping theorem.
As $\nu \mapsto I(\nu)$ is lower semicontinuous, the mapping 
$\nu \mapsto \nu \otimes \nu(U) - I(\nu)$
is upper semicontinuous and hence attains its maximum on the weakly compact set $M_P$.

We have the trivial bounds 
\begin{equation}\begin{split}
&\frac{\E[W]}{2} (\a\otimes \a) \left(e^{ \Phi(\cdot,\cdot)} \right) 
\leq F(\Phi, P, \alpha) \leq \frac{\E[W]}{2}  \Vert e^{ \Phi}\Vert, 
\end{split}\end{equation}
where the lower bound is obtained by estimating 
the sup from below by putting the spin marginals 
all equal to $\a$. The upper bound is obtained  by estimating 
the entropy from below by $0$ and using the sup-norm 
on the spin part of $U$. 

In particular 
\begin{equation}\begin{split}
&0\leq \int P (dw) S(\nu^{w,*}|\a )\leq \tilde C < \infty,
\end{split}\end{equation}
at a maximizer for the untruncated problem with some constant $\tilde C$. 
So we get by monotone convergence that 
\begin{equation}\begin{split}
&\lim_{R\uparrow \infty} \int_0^R P (dw) S(\nu^{w,*}|\a )=\int P (dw) S(\nu^{w,*}|\a ).
\end{split}\end{equation}

From that follows also that 
\begin{equation}\begin{split}
\lim_{R\uparrow\infty}&
\Biggl( \int P(dw |W\leq R)\int P(dw' |W\leq R) (\nu^{w,*}\otimes\nu^{w',*}) (U)
- \int P (dw|W\leq R) S(\nu^{w,*}|\a )\Biggr)\cr
&=\bar \nu^* \otimes \bar \nu^* (U)
- \int P (dw) S(\nu^{w,*}|\a ).
\end{split}\end{equation}

Since we have that the true sup of the truncated problem 
is bounded from below by the restricted expression taken at a maximizer 
for the untruncated problem, that is 
\begin{equation}\begin{split}
F(\Phi, P, \alpha ; R) & \geq  \int P(dw |W\leq R)\int P(dw' |W\leq R) (\nu^{w,*}\otimes\nu^{w',*}) (U) \cr
&\qquad - \int P (dw|W\leq R) S(\nu^{w,*}|\a ),
\end{split}\end{equation}
one inequality of the desired claim follows, namely 
\begin{equation}\begin{split}
&\liminf_{R \to\infty} F_R\geq F(\Phi, P , \alpha).
\end{split}\end{equation}

To get the opposite bound 
look at the true maximizers at size $R$, and use conditional measures 
$\a$ for $w\geq R$ to define conditional measures for all values of $w$: 

Let $ \nu^{R,*}$ be the maximizer for \eqref{MaxTrunc}, i.e.,
\begin{equation}
F(\Phi, P, \a; R) = \bar \nu^{R,*} \otimes \bar \nu^{R,*} (U)
- \int P (dw) S(\nu^{R,*,w}|\a ).
\end{equation}
Define the measure $\tilde \nu ^R \in \mathcal{M}_1(E \times E')$ with marginal $P$ on the weight space $E'$ and marginal on the local state space $E$ given by 
\begin{equation}\begin{split}\label{condM}
\tilde \nu^{R, w}= \begin{cases} \nu^{R, *,w} &\mbox{if } w \leq R \\ 
\a & \mbox{else. } \end{cases}
\end{split}\end{equation}
Then
\begin{equation}\begin{split}
F(\Phi, P, \alpha) &\geq \tilde \nu^{R} \otimes \tilde \nu^{R} (U) - \int P(dw) S(\tilde \nu^{R,w}|\a ) \\
&\geq \int_0^R P(dw) \int_0^R P(dw') \nu^{R, *, w} \otimes \nu^{R,*,w'} (U) - \int_0^R P(dw) S(\nu^{R,*,w} \mid \a) \\
&= F(\Phi, P, \a ; R),
\end{split}\end{equation}
for all $R>0$ and hence $F(\Phi, P, \alpha) \geq \limsup_{R\to \infty} F(\Phi, P, \a; R)$.
\end{proof}
Combining the above theorem with Proposition~\ref{pro:AnnPrRep}, proves the following corollary.
\begin{cor}\label{cor:AnnPr}
Assume that the weights $(w_i)_{i \in \mathbb{N}}$ satisfy Condition~\ref{con:weight}. Let $\Phi: E \times E \to \R$ be a bounded pair interaction potential and $\a \in \mathcal{M}_1(E)$ an a priori measure on the state space. Then the annealed pressure exists in the thermodynamic limit $N \to\infty$ and is given by
\begin{equation}\begin{split}
\psi(\Phi, P, \a) &= \lim_{N \to \infty} \frac{1}{N}  \log Q_N^w \a^N \left( e^{\sum_{(i,j) \in E_N} \Phi(\cdot, \cdot)}\right) \\
&= \sup_{\nu \in \mathcal{M}_1(E\times E'): \nu(dw) = P(dw)} \left( \nu \otimes \nu (U) - \int P(dw) S(\nu^w\mid\a) \right)< \infty,
\end{split}\end{equation}
where
\begin{equation}
U(\sigma, \sigma', w, w') = \frac{ww'}{2 \E[W]} e^{ \Phi(\sigma, \sigma')}.
\end{equation}

\end{cor}

\section{Functional mean-field equation}

\subsection{Derivation of the mean-field equation}

As we have seen in Corollary~\ref{cor:AnnPr} the pressure of the annealed model is in our setting always given by
\begin{equation}\begin{split}\label{AnnPrMax}
\psi(\Phi, P, \alpha) = \max_{ \nu(d w)=P(dw)}
\Biggl(  \nu \otimes  \nu (U)
- \int P (dw) S( \nu^{w}|\a )\Biggr).
\end{split}\end{equation}
As it turns out the stationary points of the functional which is explicitly given on the r.h.s. of \eqref{AnnPrMax} always meet a certain type of mean-field equation: 

\begin{thm}\label{thm:MF}
The maximizer on the r.h.s. of \eqref{AnnPrMax} is of the form
\begin{equation}\begin{split}
\frac{d \nu^{w}}{d \a}(\s)=\frac{d \nu^{w, V}}{d \a}(\s)
:=\frac{\exp\Bigl( w V(\s)         \Bigr) }{\int d\a(\tilde \s)\exp\Bigl(   w V(\tilde\s)         \Bigr) },
\end{split}\end{equation}
where $V: E \to \R$ satisfies
\begin{equation}\begin{split}
V(\s)=\int P^{sb}(dw )\frac{\int d\a(\tilde \s)e^{\Phi(\s,\tilde \s)}\exp\Bigl(   w V(\tilde \s)         \Bigr) }{\int d\a(\tilde \s)\exp\Bigl(  w V(\tilde\s)         \Bigr) }.
\end{split}\end{equation}
\end{thm}

\begin{rk}
The map $V$ can be interpreted as an effective mean-field potential which appears here as an order parameter. This potential has to satisfy the fixed point equation
\begin{equation}\begin{split}\label{eq:FixPoEq}
V=T_{\Phi,P,\alpha}(V),
\end{split}\end{equation}
with the non-linear map $T_{\Phi,P,\alpha}:C(E)\rightarrow C(E)$ given by 
\begin{equation}\begin{split}
T_{\Phi,P,\alpha}(V)(\s):=\int P^{sb}(dw )\int\nu^{w, V}(d\tilde \s)(e^{\Phi(\s,\tilde \s)}).
\end{split}\end{equation}
In words, this non-linear operator $T$  associates to the potential $V$ on $E$ the size-biased expectation 
of the single site expectation with exponent $w V$ of $e^{\Phi(\cdot, \cdot)}$
over one of the variables. 
\end{rk}

\begin{proof}
The equation for stationary points is obtained by describing the constrained 
infimum over $\bar \nu(d w)=P(dw)$ in terms of the conditional 
probabilities $w \mapsto \nu^w(d \s)$ and taking variations $w \mapsto \r^w(d \s)$ 
where $\r^w(d \s)$ are signed measures on the single spin space 
$E$ with mass zero. 
This means that we must have 
\begin{equation}\begin{split}
 2 \int P(dw)&\int P(dw')  \int\nu^{w'}(d\s') \int \r^w(d\s) 
U(\s,\s', w,w')\cr
& = \int P(dw) \int\r^{w}( d\s)\log \frac{d\nu^w }{d\a}(\s),
\end{split}\end{equation}
which in turn implies
\begin{equation}\begin{split}
2\int P(dw')   \int\nu^{w'}(d\s') 
U(\s,\s', w,w')
& = \log \frac{d\nu^w }{d\a}(\s) + C(w).
\end{split}\end{equation}
So we have
\begin{equation}\begin{split}
\frac{d\nu^w }{d\a}(\s)
&=\exp \Bigl( 2\int P(dw')   \int\nu^{w'}(d\s') 
U(\s,\s', w,w')- C(w)
\Bigr)\cr
&=\exp \Bigl(  w \int P^{sb}(dw')   \int\nu^{w'}(d\s') 
e^{\Phi(\s,\s')}- C(w)
\Bigr).\cr
\end{split}\end{equation}
Doing so we arrive at the mean-field equation of the form 
\begin{equation}\begin{split}
\frac{d \nu^{w}}{d \a}(\s)=\frac{d \nu^{w, V}}{d \a}(\s)
:=\frac{\exp\Bigl( w V(\s)         \Bigr) }{\int d\a(\tilde \s)\exp\Bigl(   w V(\tilde\s)         \Bigr) },
\end{split}\end{equation}
where the mapping $V:E\to \R$ must satisfy the equation 
\begin{equation}\begin{split}
V(\s)=\int P^{sb}(dw )\frac{\int d\a(\tilde \s)e^{\Phi(\s,\tilde \s)}\exp\Bigl(   w V(\tilde \s)         \Bigr) }{\int d\a(\tilde \s)\exp\Bigl(  w V(\tilde\s)         \Bigr) }.
\end{split}\end{equation}
In short 
\begin{equation}\begin{split}
V(\s)=\int P^{sb}(dw )\int\nu^{w, V}(d\tilde \s)(e^{\Phi(\s,\tilde \s)}).
\end{split}\end{equation}
\end{proof}

\subsection{Useful variational representation of the annealed pressure}

We may restrict the supremum over the (finitely many) solutions of the fixed point equation for $V$ \eqref{eq:FixPoEq} getting 
\begin{equation}\begin{split}
\psi(\Phi,P)=\sup_{V: T_{\Phi,P,\a}(V) = V}
\Biggl( \bar \nu^V \otimes \bar \nu^V (U)
- \int P (dw) S(\nu^{w,V}|\a )
\Biggr), 
\end{split}\end{equation}
where $\bar \nu^V$ is the joint measure with marginals $\nu^{w,V}$. This can be rewritten 
in terms of a nice formula making various free-energy like terms apparent as summands. 
In the following all the supremums appearing will be taken over solutions of \eqref{eq:FixPoEq}. 
We have
\begin{equation}\begin{split}\
\psi(\Phi,P)
=\sup_{V: T_{\Phi,P,\a}(V) = V}
\Biggl( \frac{\E(W)}{2}  \int P^{sb } (dw) & \int P^{sb } (dw')
\int\nu^{w,V}(d\s) \int \nu^{w',V}(d\s')  e^{\Phi(\s,\s')} \cr
&- \int P (dw) S(\nu^{w,V}|\a )
\Biggr) .
\end{split}\end{equation}

Using again the mean-field equation for $V$ on the first term we have 
\begin{equation}\begin{split}
\psi(\Phi,P)
=\sup_{V: T_{\Phi,P,\a}(V) = V}
\Biggl( \frac{\E(W)}{2} & \int P^{sb } (dw) 
\int\nu^{w,V}(d\s) V(\s) 
- \int P (dw) S(\nu^{w,V}|\a )
\Biggr) .
\end{split}\end{equation}
Using the exponential form for $\nu^{w,V}$ this is 

\begin{equation}\begin{split}
\psi(\Phi,P)
&=\sup_{V: T_{\Phi,P,\a}(V) = V}
\Biggl( \frac{\E(W)}{2}\int P^{sb } (dw) 
\int\nu^{w,V}(d\s) V(\s) \cr
&- E(W) \int P^{sb} (dw)  \nu^{w,V}(d\s)V(\s) + \int P (dw)\log \int\a(d\s)(e^{w V(\s)})
\Biggr) .
\end{split}\end{equation}

With the small benefit of the  cancellations between the first two terms we arrive at the following 
theorem. 

\begin{thm}\label{thm-pressure}
\begin{equation}\begin{split}\label{eq-general-pressure}
&\psi(\Phi,P)\cr
&=\sup_{V: T_{\Phi,P,\a}(V) = V}
\Biggl(-  \frac{\E(W)}{2} \int P^{sb} (dw) \int \nu^{w,V}(d\s)V(\s) 
+ \int P (dw)\log \int\a(d\s)(e^{w V(\s)})
\Biggr)  \cr
&=\sup_{V: T_{\Phi,P,\a}(V) = V}
\Bigl(-  \frac{\E(W)}{2} E^{sb}\bigl( \nu^{W,V}(V)\bigr)
+ E \bigl(\log \a(e^{W V})\bigr)
\Bigr).  \cr
\end{split}\end{equation}
\end{thm}

\begin{rk}
Interestingly, there is no explicit $\Phi$-dependence 
in this formula. However, remember that the $\sup$ has to be taken 
only w.r.t.\ to $V$'s which are solutions to the mean field equation. 
\end{rk}

\subsection{A criterion for uniqueness}

Let us start the discussion by providing a criterion for uniqueness 
based on smallness of the interaction potential $\Phi$ (which could be scaled 
with a prefactor $\b$, letting $\b$ tend to zero), and smallness of the weights 
in some sense. 

To do so, we put the sup-norm on $E$ und use the Banach fixed-point theorem 
on the space of single-site potentials on $E$. 
We write 
\begin{equation}\begin{split}
T_{\Phi,P,\alpha}(V)(\s)- T_{\Phi,P,\alpha}(\bar V)(\s)
&=\int P^{sb}(dw )\int_0^1 \frac{d}{dt}\int\nu^{w, V+ t(\bar V- V)}(d\tilde \s)(e^{\Phi(\s,\tilde \s)}) dt \cr
&=\int P^{sb}(dw )w\int_0^1 \text{cov}_{\nu^{w, V+ t(\bar V- V)}}
(e^{\Phi(\s,\cdot)}; \bar V-V)dt. \cr
\end{split}\end{equation}
Define the variation of a function as $\d(f):=\sup_{\o,\o'}(f(\o)-f(\o'))$. Then, using the above,
\begin{equation}\begin{split}
&\delta(T_{\Phi,P,\alpha}(V)(\s)- T_{\Phi,P,\alpha}(\bar V))\cr
&=\sup_{\s_1,\s_2} \int P^{sb}(dw )w\int_0^1 \text{cov}_{\nu^{w, V+ t(\bar V- V)}}
(e^{\Phi(\s_1,\cdot)}-e^{\Phi(\s_2,\cdot)}; \bar V-V)dt \cr
&=\sup_{\s_1,\s_2} \int P^{sb}(dw )w\int_0^1 \text{cov}_{\nu^{w, V+ t(\bar V- V)}}
(e^{\Phi(\s_1,\cdot)}-e^{\Phi(\s_2,\cdot)}; \bar V-V+a)dt, \cr
\end{split}\end{equation}
for any constant $a$. Estimating the covariance uniform in the measure gives 
us
\begin{equation}
\delta(T_{\Phi,P,\alpha}(V)(\s)- T_{\Phi,P,\alpha}(\bar V))
\leq \Vert V-\bar V +a\Vert \delta^{(2)}(e^\Phi) \int P^{sb}(dw )w,
\end{equation}
where the second order variation of a function of two parameters is defined as
\begin{equation}
\delta^{(2)}(f) = \sup_{\s_1,\s_2,\s_3,\s_4} \big( f(\s_1,\s_3)-f(\s_2,\s_3) - (f(\s_1,\s_4)-f(\s_2,\s_4))\bigr).
\end{equation}
Since this inequality holds for all constants $a$, we can take the infimum over $a$, and use that $\inf_a \Vert f+a\Vert = \frac12\delta(f)$, to get the following theorem:

\begin{thm}\label{thm:unique} 
There is a unique solution to the fixed point equation \eqref{eq:FixPoEq}
for $V$ and hence absence of phase transition in the annealed spin model on 
generalized random graphs  
whenever 
\begin{equation}\begin{split}\label{eq:unique}
&E^{sb}(W) \,\,
\frac12 \d^{(2)} ( e^{\Phi})< 1. \cr
\end{split}\end{equation}
\end{thm}

\subsection{Ising model}
The Ising model is defined by taking $E=\{-1,+1\}$, $\Phi(\s,\s')=\b \s \s' $ and  $\a$ is symmetric Bernoulli, that is, $\alpha=\frac12 \delta_{-1}+\frac12 \delta_{+1}$. Any $V(\s)$ has to be of the form $V(\s)=v \s+\l$. Plugging this into~\eqref{eq:FixPoEq} gives
\begin{equation}\begin{split}\
\int P^{sb}(dw )\frac{\cosh( \b \s  + w v  ) 
}{
\cosh(   w v  ) }
 = v\s +\l,
\end{split}\end{equation}
which is equivalent to
\begin{equation}\begin{split}
\int P^{sb}(dw )\frac{\cosh( \b \s)\cosh( w v  )+  \sinh( \b \s)\sinh( w v  )
}{
\cosh(   w v  ) }
 = v\s +\l.
\end{split}\end{equation}
Using that $\s$ can only take values  $-1$ and $+1$, we can write this as
\begin{equation}\begin{split}
\cosh( \b )+ \s  \sinh( \b )\int P^{sb}(dw ) \tanh( w v  )
 = v\s +\l.
\end{split}\end{equation}
Since $\l$ is arbitrary, we have the fixed point equation for the effective potential 
\begin{equation}\begin{split}\label{eq-fixedpointIsing}
T(v) := \sinh( \b )\int P^{sb}(dw ) \tanh( w v  )
 = v .
\end{split}\end{equation}
There is symmetry breaking if we have a solution $v>0$. Note that 
\begin{equation}
T'(v) = \sinh( \b )\int P^{sb}(dw ) w (1-\tanh^2( w v  )), 
\end{equation}
and 
\begin{equation}
T''(v) = -2 \sinh( \b )\int P^{sb}(dw ) w^2 \tanh(w v) (1-\tanh^2( w v  ))<0,
\end{equation}
for $v>0$. Hence, $T(v)$ is a concave function for $v>0$, and hence there is a solution $v>0$ iff $T'(0)>1$, i.e.,
\begin{equation}
\sinh( \b )\int P^{sb}(dw )w = \sinh(\b) \frac{\E[W^2]}{\E[W]} > 1.
\end{equation}
Hence, we have a phase transition at
\begin{equation}
\beta_c = \asinh\frac{\E[W]}{\E[W^2]}.
\end{equation}
Indeed, the fixed point equation~\eqref{eq-fixedpointIsing} is the same as the one obtained in~\cite{GiaGibHofPri16}, see also~\cite{DomGiaGibHofPri16} and Proposition~\ref{pro-ising} below for an analysis of the critical behavior of this model.

\section{Critical behavior of rank-$2$ continuous transition kernels \label{sec-criticalbehavior}} 
\subsection{State space $E=[-1,1]$  \label{sec-criticalbehavior5.1}} 
Let $\a$ be any a priori probability measure on $E$ and $g$ a function on $E$.
Define the kernel $K(\s,\s')= e^{\Phi(\s,\s')}=c+ \theta g(\s)g(\s')$ where we assume that the constants $c$ and $\theta$ can be chosen such that $K$ is strictly positive. 
This kernel has its range spanned by the constants $c$ and the function $g$ on $E$. Therefore we call this a {\em rank-$2$ continuous transition kernel}. 

Observe that 
\begin{equation}\begin{split}
T_{\Phi,P,\alpha}(V)(\s)=\int P^{sb}(dw )\int\nu^{w, V}(d\tilde \s)(K(\s,\tilde \s))= c + \theta g(\s) \E^{sb} \nu^{W, V}(g).
\end{split}\end{equation}
Hence any $V$ satisfying 
\begin{equation}\begin{split}
V=T_{\Phi,P,\alpha}(V),
\end{split}\end{equation}
must be of the form 
$V(\s)=c + m g(\s)$. Hence the fixed point equation reduces to 
the one-dimensional equation 

\begin{equation}\begin{split}\label{eq-fixedpointm}
m/\theta= \E^{sb}\nu^{W, m g}(g)=:\varphi(m).
\end{split}\end{equation}

Denote by $\alpha_0$ an even measure on $[-1,1]$. Note that for $\alpha=\alpha_0$, we always have the solution $\varphi(0)=0/\theta$. Furthermore, we have the following phase transition if $\alpha=\alpha_0$:
\begin{lem}\label{lem-concavity}
Let $\alpha=\alpha_0$ for some even measure $\alpha_0$ and $g$ be an odd function. Suppose that, for all $t>0$,
\begin{equation}\label{eq-concavity-assumption}
\frac{\partial^3}{\partial t^3} \log \alpha_0(e^{t g}) <0.
\end{equation}
Then, with
\begin{equation}
1/\theta_c = \E^{sb}(W)\alpha_0(g^2),
\end{equation}
there exists no positive solution to $m/\theta=\varphi(m)$ if $\theta\leq\theta_c$ and a unique positive solution if $\theta>\theta_c$.
\end{lem}
\begin{proof}
Note that for $m>0$
\begin{equation}
\varphi''(m) = \E^{sb}\left[\frac{\partial^2}{\partial m^2} \nu^{W, m g}(g) \right] =  \E^{sb}\left[W^2 \frac{\partial^3}{\partial t^3} \log \alpha_0(e^{t g})\Big|_{t=W m}\right] <0.
\end{equation}
Hence, $\varphi$ is a concave function for $m\geq 0$ and hence can have at most one positive solution to $m/\theta=\varphi(m)$. Such a solution exists iff $\varphi'(0) > 1/\theta$. Since $g$ is odd,
\begin{equation}
\varphi'(0) = \E^{sb}\left[\frac{\partial}{\partial m} \nu^{W, m g}(g) \Big|_{m=0} \right] =  \E^{sb}\left[W \frac{\partial^2}{\partial t^2} \log \alpha_0(e^{t g})\Big|_{t=0}\right] =\E^{sb}(W)\alpha_0(g^2).
\end{equation} 
Hence, the condition $\varphi'(0) > 1/\theta$ indeed corresponds to $\theta>\theta_c$.
\end{proof}
By symmetry the same holds for the number of negative solutions. If there is a positive solution $m^+$, then also $m^-=-m^+$ is a solution. It turns out that if a positive solution $m^+$ exists, this also gives the maximizer of~\eqref{eq-general-pressure}:
\begin{thm}\label{thm-mplusismaximizerh0}
Let $\alpha=\alpha_0$ for some even measure $\alpha_0$ and let $g$ be an odd function on $E$. Suppose that, for all $t>0$,
\begin{equation}\label{eq-concavity-assumption}
\frac{\partial^3}{\partial t^3} \log \alpha_0(e^{t g}) <0.
\end{equation}
Then,
\begin{equation}
\psi(\Phi,P) = \left\{\begin{array}{ll} \frac{c}{2}\E(W), & {\rm if\ } \theta\leq\theta_c,\\
		\frac{c}{2}\E(W) -\frac{1}{2}\E(W) \frac{(m^+)^2}\theta + \E \log \alpha_0(e^{W g m^+}), &{\rm if\ } \theta>\theta_c, \end{array}\right.
\end{equation}
where $m^+$ is the unique positive solution to $m/\theta=\varphi(m)$.
\end{thm}
\begin{proof}
Recall that $V=c+gm$, so that, for any solution to $m/\theta=\varphi(m)$,
\begin{equation}\begin{split}
-  \frac{\E(W)}{2}  E^{sb}\bigl( \nu^{W,V}(V)\bigr) &=-  \frac{\E(W)}{2}  E^{sb}\bigl( \frac{\alpha((c+gm)e^{W(c+gm)})}{\alpha(e^{W(c+gm)})}\bigr) \cr
&=-\frac{c}{2}\E(W) -  \frac{\E(W)}{2} m E^{sb}\bigl( \nu^{W,gm}(g)\bigr)\cr
&= -\frac{c}{2}\E(W) - \frac{\E(W)}{2} \frac{m^2}{\theta},
\end{split}\end{equation}
where we used the fixed point equation in the last equality. Furthermore,
\begin{equation}
E \bigl(\log \a(e^{W V})\bigr) = E \bigl(\log \a(e^{W(c+gm)})\bigr) = c\E(W) + E \bigl(\log \a(e^{Wgm})\bigr).
\end{equation}
Hence, it follows from Theorem~\ref{thm-pressure} that
\begin{equation}
\psi(\Phi,P) = \sup_m \left(\frac{c}{2}\E(W) -\frac{1}{2}\E(W) \frac{m^2}\theta + \E \log \alpha(e^{W g m})\right)=: \sup_m \psi(m),
\end{equation} 
where the the  supremum is over all solutions to $m/\theta=\varphi(m)$. If $\theta \leq \theta_c$ there is only one solution $m=0$ and we are done. 

For $\theta>\theta_c$, we need to compare the values of $\psi(0)$ and $\psi(m^+)$. (The negative solution gives the same value as $\psi(m^+)$ by symmetry). Note that $\psi'(0)=0$ and
\begin{equation}
\psi''(0) = \E(W) \left(\phi'(0) - \frac{1}{\theta} \right),
\end{equation}
which has been shown to be negative for $\theta>\theta_c$ in the previous lemma, so that $m=0$ is a local minimum. Furthermore,
$\psi'(m^+)=0$ and there are no other positive values for which this holds and
\begin{equation}
\lim_{|m|\to\infty} \psi(m) = -\infty.
\end{equation}
Hence, we have that $\psi$ must have a global maximum at $m^+$.
\end{proof}

We can break the symmetry of the system by adding a magnetic field by introducing an exponential tilting of the measure $\alpha_0$, that is, for some $h\in\mathbb{R}$, we let
\begin{equation}
\alpha(d\s) = \frac{e^{h g(\s)}\alpha_0(d\s)}{\int e^{h g(\s')}\alpha_0(d\s')}.
\end{equation}
If we furthermore suppose that $\sgn(g(\sigma))=\sgn(\sigma)$, then $h>0$ introduces a positive bias to the spin values. In this case, there is only one nonnegative solution to $m/\theta=\phi(m)$  and this is also the maximizer of $\psi(m)$ as the next theorem shows:
\begin{thm}
Suppose that $g$ is an odd function with $\sgn(g(\sigma))=\sgn(\sigma)$ and that, for all $t>0$,
\begin{equation}\label{eq-concavity-assumption}
\frac{\partial^3}{\partial t^3} \log \alpha_0(e^{t g}) <0.
\end{equation}
If $h>0$, then
\begin{equation}
\psi(\Phi,P) = \frac{c}{2}\E(W) -\frac{1}{2}\E(W) \frac{(m^+)^2}\theta + \E \log \alpha(e^{W g m^+}),
\end{equation}
where $m^+$ is the unique nonnegative solution to $m/\theta=\varphi(m)$.
\end{thm}
\begin{proof}
We first prove that there indeed is a unique nonnegative solution to $m/\theta=\varphi(m)$. As in Lemma~\ref{lem-concavity}, we have that for all $m>0$
\begin{equation}
\varphi''(m) = \E^{sb}\left[\frac{\partial^2}{\partial m^2} \frac{\alpha_0(g e^{(Wm+h)g})}{\alpha_0(e^{(Wm+h)g})} \right] =  \E^{sb}\left[W^2 \frac{\partial^3}{\partial t^3} \log \alpha_0(e^{t g})\Big|_{t=W m+h}\right] <0,
\end{equation}
so that again $\varphi(m)$ is concave for $m\geq0$. For $m=0$ we have that
\begin{equation}
\varphi(0)= \E^{sb}\left[\nu^0(g)\right] =  \frac{\alpha_0(g e^{hg})}{\alpha_0(e^{hg})}.
\end{equation}
Using that $g$ is an odd function and $\sgn(g(\sigma))=\sgn(\sigma)$, we can compute
\begin{equation}\begin{split}
\alpha_0(g e^{hg}) &= \int_{-1}^{0} \alpha_0(d\s) g(\s) e^{hg(\s)}+\int_{0}^{1} \alpha_0(d\s) g(\s) e^{hg(\s)} \cr
&=\int_{0}^{1} \alpha_0(d\s) g(\s) \left(e^{hg(\s)}-e^{-hg(\s)}\right) >0,
\end{split}\end{equation}
since $h>0$. Hence, also $\varphi(0)>0$ and there is a unique positive fixed point.

Using similar arguments, one can show that there are either zero negative fixed points, in which case we are done, or there are two negative solutions, we call the smallest of these two $m^-$. As in Theorem~\ref{thm-mplusismaximizerh0}, one can show that $m^-$ is the only other candidate for the maximizer. We compare the two values by computing
\begin{equation}
\frac{\partial}{\partial h} \left(\psi(m^+)-\psi(m^-)\right),
\end{equation}
and show that this is positive for all $h$. For $m$ any fixed point, we have
\begin{equation}
\frac{\partial}{\partial h}\psi(m) = \psi'(m)\frac{\partial m}{\partial h} + \E\left[\frac{\alpha_0(g e^{(Wm+h)g})}{\alpha_0(e^{(Wm+h)g})}\right] - \E\left[\frac{\alpha_0(g e^{hg})}{\alpha_0(e^{hg})}\right].
\end{equation}
Note that $\psi'(m)=0$, since $m$ is a fixed point. Also,
\begin{equation}
\frac{\partial}{\partial m} \E\left[\frac{\alpha_0(g e^{(Wm+h)g})}{\alpha_0(e^{(Wm+h)g})}\right] =  \E\left[W\frac{\alpha_0(g^2 e^{(Wm+h)g})}{\alpha_0(e^{(Wm+h)g})}-W \left(\frac{\alpha_0(g e^{(Wm+h)g})}{\alpha_0(e^{(Wm+h)g})}\right)^2\right] >0.
\end{equation}
Hence,
\begin{equation}
\frac{\partial}{\partial h} \left(\psi(m^+)-\psi(m^-)\right) = \E\left[\frac{\alpha_0(g e^{(Wm^++h)g})}{\alpha_0(e^{(Wm^++h)g})}\right]-\E\left[\frac{\alpha_0(g e^{(Wm^-+h)g})}{\alpha_0(e^{(Wm^-+h)g})}\right] >0.
\end{equation}
\end{proof}

We now analyze the critical behavior of $m^+$ around the critical point ${(\theta,h)=(\theta_c,0)}$. For this we write that $f(x)\asymp g(x)$ if $f(x)/g(x)$ is bounded away from $0$ and $\infty$ for the specified limit. We define the critical exponents $\boldsymbol{\beta}$ and $\boldsymbol{\delta}$ as
\begin{align}
m^+(\theta, 0) &\asymp (\theta-\theta_c)^{\boldsymbol{\beta}},  && {\rm for\ } \theta \searrow \theta_c;\\
m^+(\theta_c, h) &\asymp h^{1/\boldsymbol{\delta}},  && {\rm for\ } h \searrow 0;
\end{align}
respectively. Before we compute these critical exponents, we first show that the phase transition is continuous:
\begin{lem}\label{lem-contphasetrans}
Suppose that $g$ is an odd function with $\sgn(g(\sigma))=\sgn(\sigma)$ and that, for all $t>0$,
\begin{equation}
\frac{\partial^3}{\partial t^3} \log \alpha_0(e^{t g}) <0
\end{equation}
where we assume that the measure $\a_0$ is even.
Then, it holds that
\begin{equation}
\lim_{\theta \searrow \theta_c} m^+(\theta,0) = 0, \quad {\rm and} \quad \lim_{h \searrow 0} m^+(\theta_c,h)=0.
\end{equation}
\end{lem}
\begin{proof}
By assumption, for some $\xi\in(0,t)$,
\begin{equation}\label{eq-strictlylesslinear}
\frac{\partial}{\partial t}\log \alpha_0(e^{tg}) = \alpha_0(g^2)t+ \frac{1}{2} \frac{\partial^3}{\partial t^3} \log \alpha_0(e^{t g})\Big|_{t=\xi} t^2 < \alpha_0(g^2)t.
\end{equation}
Hence,
\begin{equation}
m^+(\theta,h) = \theta \E^{sb}\left[\frac{\partial}{\partial t}\log \alpha_0(e^{tg})\Big|_{t=W m^+(\theta,h) +h}\right] < \theta \alpha_0(g^2) (\E^{sb}(W)m^+(\theta,h)+h).
\end{equation}
Now, suppose that $\lim_{\theta \searrow \theta_c} m^+(\theta,0) = c>0$. Then, it follows from the above that
\begin{equation}
c= \lim_{\theta \searrow \theta_c} m^+(\theta,0) <\lim_{\theta \searrow \theta_c}  \theta \alpha_0(g^2) \E^{sb}(W)m^+(\theta,0)=c,
\end{equation}
which leads to a contradiction. Similarly, if we suppose that $\lim_{h \searrow 0} m^+(\theta_c,h) = c>0$, then
\begin{equation}
c= \lim_{h \searrow 0} m^+(\theta_c,h) <\lim_{h \searrow 0}  \theta_c \alpha_0(g^2) (\E^{sb}(W)m^+(\theta_c,h)+h)=c,
\end{equation}
again leading to a contradiction.
\end{proof}

The behavior at criticality depends sensitively on whether or not a certain moment of the weight variable is finite or not. When this is not the case we will assume that the weight variable at least meets a power-law bound for its tail:

\begin{con}\label{con:powerweight}
Let $k$ be some natural number. The weight variable $W$ satisfies either of the following properties:
\begin{enumerate}[(i)]
\item
$\E \left[ W^k\right] < \infty$,
\item
$W$ obeys a power law with exponent $\tau \in (3, k+1]$, i.e. there exist constants $C_W > c_W >0$ and $w_0 >1$ such that
\begin{equation}
c_W w^{-(\tau - 1)} \leq \mathbb{P}(W > w) \leq C_W w^{-(\tau - 1)}, \qquad \forall w>w_0.
\end{equation}
\end{enumerate}
\end{con}

\begin{thm}\label{thm-critexp} 
Suppose that $g$ is an odd function with $\sgn(g(\sigma))=\sgn(\sigma)$ and that, for all $t>0$,
\begin{equation}
\frac{\partial^3}{\partial t^3} \log \alpha_0(e^{t g}) <0
\end{equation}
where $\a_0$ is assumed to be an even measure.
Let $k$ be the smallest natural number such that 
\begin{equation}\label{eq-smallestknegative}
\frac{\partial^k}{\partial t^k}\log \alpha_0(e^{tg})\Big|_{t=0}<0
\end{equation}
and assume that the weight variable $W$ satisfies one of the properties of Condition \ref{con:powerweight} for this $k$.
Then,
the critical exponents $\boldsymbol{\beta}$ and $\boldsymbol{\delta}$ exist and satisfy
\begin{center}
{\renewcommand{\arraystretch}{1.2}
\renewcommand{\tabcolsep}{1cm}
\begin{tabular}[c]{c|ccc}
 &  $\tau\in(3,k+1)$ & $\E(W^k)<\infty$    \\
\hline
$\boldsymbol{\beta}$  & $1/(\tau-3)$ & $1/(k-2)$ \\
$\boldsymbol{\delta}$ & $\tau-2$ & $k-1$
\end{tabular}
}
\end{center}
For the boundary case $\tau=k+1$ there are the following logarithmic corrections for $\boldsymbol{\beta}=1/(k-2)$:	\begin{equation}
	\label{log-corr-M-tau5}
	m^+(\theta,0) \asymp \Big(\frac{\theta-\theta_c}{\log{1/(\theta-\theta_c)}}\Big)^{1/(k-2)} \quad {\rm for\ } \theta \searrow \theta_c,
\end{equation}
and $\boldsymbol{\delta}=k-1$:
\begin{equation}
m^+(\theta_c,h) \asymp \Big(\frac{h}{\log(1/h)}\Big)^{1/(k-1)} \quad {\rm for\ } h \searrow 0.
\end{equation}
\end{thm}
Plots of these critical exponents for $k=4$ and $k=6$ can be found in Figure~\ref{fig-critexps}. Note that for $k=4$, these values are the same as for both the annealed and quenched Ising model~\cite{DomGiaGibHofPri16,DomGiaHof14}. Also the proof is similar as we show now.

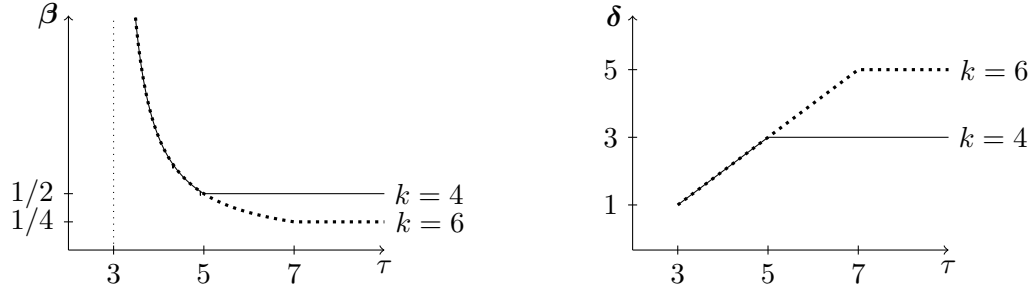
\begin{figure}
\begin{center}
\begin{tikzpicture}[scale=0.6]
	\draw[->] (0,0) -- (7,0) node[anchor=north] {$\tau$};
	\draw (1,0.1) -- (1,-0.1) node[anchor=north] {$3$};
	\draw (3,0.1) -- (3,-0.1) node[anchor=north] {$5$};
	\draw (5,0.1) -- (5,-0.1) node[anchor=north] {$7$};
	\draw[->] (0,0) -- (0,5.2) node[anchor=east] {$\boldsymbol{\beta}$};
	\draw (.1,1.25) -- (-.1,1.25) node[anchor=east] {$1/2$};
	\draw (.1,.625) -- (-.1,.625) node[anchor=east] {$1/4$};
	\draw[dotted] (1,0) -- (1,5.2);
	\draw[domain=1.485:3, samples=500] plot(\x, {2.5/(\x-1)}) -- (3,1.25) -- (7,1.25) node[anchor=west] {$k=4$};
	\draw[very thick, dotted, domain=1.485:5, samples=500] plot(\x, {2.5/(\x-1)}) -- (5,.625) -- (7,.625) node[anchor=west] {$k=6$};
	
\begin{scope}[xshift=12.5cm]
	\draw[->] (0,0) -- (7,0) node[anchor=north] {{$\tau$}};
	\draw (1,0.1) -- (1,-0.1) node[anchor=north] {$3$};
	\draw (3,0.1) -- (3,-0.1) node[anchor=north] {$5$};
	\draw (5,0.1) -- (5,-0.1) node[anchor=north] {$7$};
	\draw[->] (0,0) -- (0,5.2) node[anchor=east] {$\boldsymbol{\delta}$};
	\draw (.1,1) -- (-.1,1) node[anchor=east] {$1$};
	\draw (.1,2.5) -- (-.1,2.5) node[anchor=east] {$3$};
	\draw (.1,4) -- (-.1,4) node[anchor=east] {$5$};
	\draw (1,1) -- (3,2.5) -- (7,2.5) node[anchor=west] {$k=4$};
	\draw[very thick, dotted] (1,1) -- (5,4) -- (7,4) node[anchor=west] {$k=6$};

\end{scope}
\end{tikzpicture}
\end{center}
\caption{Plots of the critical exponents $\boldsymbol{\beta}$ (left) and $\boldsymbol{\delta}$ (right) for $k=4$ (solid) and $k=6$ (dotted).}\label{fig-critexps}
\end{figure}

\begin{proof}
Observe that, since $m^+$ satisfies $m^+/\theta=\varphi(m^+)$,
\begin{equation}\label{eq-mincumgen}\begin{split}
\frac{m^+}{\theta} &= \varphi(m^+)= \E^{sb}\left[\nu^{W,gm^+}(g)\right] = \E^{sb}\left[\frac{\alpha_0(g e^{Wgm^++hg})}{\alpha_0(e^{Wgm^++hg})}\right] \cr
&= \E^{sb}\left[\frac{\partial}{\partial t}\log \alpha_0(e^{tg})\Big|_{t=Wm^++h}\right].\cr
\end{split}\end{equation}
Since the phase transition is continuous by Lemma~\ref{lem-contphasetrans}, we have that $Wm^++h \to 0$ a.s.\ in the limits of interest. Hence, we do a Taylor expansion around $t=0$. We repeatedly use that $\alpha_0(g^{2\ell+1})=0$ for $\ell\in\mathbb{N}$ since $g$ is odd and $\alpha_0$ is even. Also note that $k$ is the smallest number bigger than $3$ for which 
$\frac{\partial^k}{\partial t^k}\log \alpha_0(e^{tg})\Big|_{t=0} \neq 0$, since by our assumption on the third derivative the first non-zero value has to be negative. Hence,
\begin{equation}\label{eq-taylorddtlog}
\frac{\partial}{\partial t}\log \alpha_0(e^{tg}) = \alpha_0(g^2) t + \frac{\partial^k}{\partial t^k}\log \alpha_0(e^{tg})\Big|_{t=\xi} \frac{t^{k-1}}{(k-1)!},
\end{equation}
for some $\xi\in(0,t)$. 
By adding and subtracting the linear term in~\eqref{eq-mincumgen}, we can write
\begin{equation}\label{eq-splitofflinear}
\frac{m^+}{\theta}=\alpha_0(g^2) h + \alpha_0(g^2)\E^{sb}(W)m^+ + \E^{sb}\left[\frac{\partial}{\partial t}\log \alpha_0(e^{tg})\Big|_{t=Wm^++h}-\alpha_0(g^2)(W m^++h)\right].
\end{equation}

Since $\frac{\partial^k}{\partial t^k}\log \alpha_0(e^{tg})\Big|_{t=0}<0$, there exists a constant $t_1$ close enough to $0$ such that there exist constants $0<c_1\leq c_2<\infty$ such that, for all $0<\xi\leq t_1$,
\begin{equation}
-c_2 \leq \frac{\partial^k}{\partial t^k}\log \alpha_0(e^{tg})\Big|_{t=\xi} \leq -c_1.
\end{equation}
Suppose now in the remainder of the proof that $0\leq h<t_1/2$. 

We first compute an upper bound on $m^+$. By~\eqref{eq-strictlylesslinear},
\begin{equation}\begin{split}
\E^{sb}&\left[\frac{\partial}{\partial t}\log \alpha_0(e^{tg})\Big|_{t=Wm^++h}-\alpha_0(g^2)(W m^++h)\right] \cr
&\leq \E^{sb}\left[\left(\frac{\partial}{\partial t}\log \alpha_0(e^{tg})\Big|_{t=Wm^++h}-\alpha_0(g^2)(W m^++h)\right)\indic{W\leq \frac{t_1}{2m^+}}\right].
\end{split}\end{equation}
By~\eqref{eq-taylorddtlog}, this equals, for some $\xi\in(0,Wm^++h)$, 
\begin{equation}\begin{split}
\E^{sb}&\left[\left(\frac{\partial^k}{\partial t^k}\log \alpha_0(e^{tg})\Big|_{t=\xi}\frac{(W m^++h)^{k-1}}{(k-1)!}\right)\indic{W\leq \frac{t_1}{2m^+}}\right] \cr
&\leq -\frac{c_1}{(k-1)!} \E^{sb}\left[(W m^++h)^{k-1} \indic{W\leq \frac{t_1}{2m^+}}\right] \cr
&\leq  -\frac{c_1}{(k-1)! \E(W)} (m^+)^{k-1} \E\left[W^k\indic{W\leq \frac{t_1}{2m^+}}\right].
\end{split}\end{equation}
For $E(W^k)<\infty$, this gives the bound, for some constant $C_1$ which may change from line to line,
\begin{equation}\label{eq-boundm}
\frac{m^+}{\theta} \leq \alpha_0(g^2) h + \alpha_0(g^2)\E^{sb}(W)m^+ - C_1 m^{+\boldsymbol{\delta}}.
\end{equation}
For $\tau\in(3,k+1)$, we use~\cite[Lemma~3.2]{DomGiaGibHofPri16} on truncated moments to obtain 
\begin{equation}
(m^+)^{k-1} \E\left[W^k \indic{W\leq \frac{t_1}{2m^+}}\right] \geq (m^+)^{k-1} c_{k,\tau} \left(\frac{t_1}{2m^+}\right)^{k-(\tau-1)} = C_1 (m^+)^{\tau-2},
\end{equation}
so that~\eqref{eq-boundm} also holds in this  case. For $\tau=k+1$ using the same lemma we obtain
\begin{equation}
\frac{m^+}{\theta} \leq \alpha_0(g^2) h + \alpha_0(g^2)\E^{sb}(W)m^+ - C_1 m^{+k-1} \log(1/m^+).
\end{equation}
Using that $1/\theta_c = \alpha_0(g^2)\E^{sb}(W)$, it follows from~\eqref{eq-boundm} that
\begin{equation}
C_1 m^+(\theta_c,h)^{\boldsymbol{\delta}} \leq \alpha_0(g^2) h,
\end{equation}
so that
\begin{equation}
\limsup_{h\searrow 0}\frac{m^{+}(\theta_c,h)}{h^{1/\boldsymbol{\delta}}} \leq \left(\frac{\alpha_0(g^2)}{C_1}\right)^{1/\boldsymbol{\delta}}<\infty.
\end{equation}
For $\theta>\theta_c$, we have that $m^+(\theta,0)>0$, and hence we can divide both sides of~\eqref{eq-boundm} to  obtain
\begin{equation}
\frac{1}{\theta} \leq  \frac{1}{\theta_c} - C_1 m^{+}(\theta,0)^{\boldsymbol{\delta}-1}.
\end{equation}
or equivalently
\begin{equation}
m^{+}(\theta,0) \leq \left(\frac{1}{C_1 \theta \theta_c}\right)^{\boldsymbol{\beta}}(\theta -  \theta_c)^{\boldsymbol{\beta}},
\end{equation}
where we used that $1/(\boldsymbol{\delta}-1)=\boldsymbol{\beta}$. Hence,
\begin{equation}
\limsup_{\theta \searrow \theta_c}\frac{m^{+}(\theta,0)}{(\theta -  \theta_c)^{\boldsymbol{\beta}}}<\infty.
\end{equation}
The proof for $\tau=5$ is similar.

For the lower bound, we again start from~\eqref{eq-splitofflinear}. We can split
\begin{equation}\begin{split}\label{eq-splitsmallbigW}
\E^{sb}&\left[\frac{\partial}{\partial t}\log \alpha_0(e^{tg})\Big|_{t=Wm^++h}-\alpha_0(g^2)(W m^++h)\right] \cr
&=\E^{sb}\left[\left(\frac{\partial}{\partial t}\log \alpha_0(e^{tg})\Big|_{t=Wm^++h}-\alpha_0(g^2)(W m^++h)\right)\left(\indic{W\leq \frac{t_1}{2m^+}}+\indic{W> \frac{t_1}{2m^+}}\right)\right].
\end{split}\end{equation}
Similarly to the upper bound we can bound the small $W$ term from below by
\begin{equation}
-\frac{c_2}{(k-1)!} \E^{sb}\left[(W m^++h)^{k-1} \indic{W\leq \frac{t_1}{2m^+}}\right] \geq -C_2 m^{+\boldsymbol{\delta}}-error,
\end{equation}
where we again used~\cite[Lemma~3.2]{DomGiaGibHofPri16} and $error$ is an error term that indeed can be shown to be negligible in the limits of interest. For the second term in~\eqref{eq-splitsmallbigW} and $E(W^k)<\infty$ it suffices to observe that 
\begin{equation}
\E^{sb}\left[(W m^++h)^{k-1} \indic{W> \frac{t_1}{2m^+}}\right] \to 0,
\end{equation}
in both limits of interest because $m^+\to0$ by Lemma~\ref{lem-contphasetrans}. For $\tau\in(3,k+1)$, we bound
\begin{equation}\begin{split}
\E^{sb}&\left[\left(\frac{\partial}{\partial t}\log \alpha_0(e^{tg})\Big|_{t=Wm^++h}-\alpha_0(g^2)(W m^++h)\right)\indic{W> \frac{t_1}{2m^+}}\right] \cr
&\geq - \alpha_0(g^2) \E^{sb}\left[(W m^++h)\indic{W> \frac{t_1}{2m^+}}\right] \cr
&\geq - \alpha_0(g^2) m^+  \left(\frac{t_1}{2m^+}\right)^{2-(\tau-1)}-error=-C_2 m^{+\boldsymbol{\delta}}-error,
\end{split}\end{equation}
where we used~\cite[Lemma~3.2]{DomGiaGibHofPri16} once more in the third line.

Hence, we have in both cases that
\begin{equation}\label{eq-lbboundm}
\frac{m^+}{\theta} \geq \alpha_0(g^2) h + \alpha_0(g^2)\E^{sb}(W)m^+ - C_2 m^{+\boldsymbol{\delta}}-error.
\end{equation}
The rest of the analysis can be done as in the upper bound.
\end{proof}

We now present several examples to which this theorem can be applied. We start with two examples for $k=4$, for which the mean field exponents are valid if $\E[W^4]<\infty$. We start by rederiving the results in~\cite{DomGiaHof14} for the Ising model.
\begin{pro}[Ising model]\label{pro-ising}
Let $\alpha_0=\frac12 \delta_{-1}+\frac12 \delta_{+1}$ and  $g(\s)=\s$. Then Theorem~\ref{thm-critexp} holds with $k=4$.
\end{pro}
\begin{proof}
Note that this is indeed an Ising model, because we can write
\begin{equation}
K(\s,\s')= c + \theta \s\s' = \tilde{c}e^{\beta \s\s'},
\end{equation}
by choosing $\tilde{c}=\sqrt{c^2-\theta^2}$ and $\beta=\frac12 \log\left(\frac{c+\theta}{c-\theta}\right)$. This can easily be checked by checking the only possible values for the spins.
For this model
\begin{equation}
\log \alpha_0(e^{t g}) = \log\frac{e^t+e^{-t}}{2}=\log\cosh t.
\end{equation}
Hence,
\begin{equation}
\frac{\partial^3}{\partial t^3} \log \alpha_0(e^{t g}) =-2 \tanh t(1-\tanh^2 t) <0,
\end{equation}
for $t>0$, and
\begin{equation}
\frac{\partial^4}{\partial t^4}\log \alpha_0(e^{tg})\Big|_{t=0}=-2(1-\tanh^2 t)^2+4\tanh^2t(1-\tanh^2 t)\Big|_{t=0}=-2.
\end{equation}
\end{proof}

We next give an example for which the spins are continuous.
\begin{pro}[Beta distribution]\label{pro:beta}
For some $b>0$, $\alpha_0$ has density
\begin{equation}
\alpha_0(d \s) = \frac{1}{2 B(b,b)} \left(\frac{1+\s}{2}\right)^{b-1} \left(1-\frac{1+\s}{2}\right)^{b-1} d\s,
\end{equation}
where $B$ is the beta function $B(b,b)=\frac{\Gamma(b)^2}{\Gamma(2b)}$ and $g(\s)=\s$. Then Theorem~\ref{thm-critexp} holds with $k=4$.
\end{pro}
\begin{proof}

The measure $\a_0$ has the density of a Beta distribution with parameters $a = b$ which is stretched out over the interval $[-1,1]$. Let $Y = 2\tilde Y-1$ with $\tilde Y \sim \text{Beta}(b,b)$. 
 Then $Y \sim \a_0$ and we obtain for the first derivative of the log-moment generating function
\begin{equation}\begin{split}
\frac{\partial}{\partial t} \log \a_0(e^{t\s}) &= \frac{\partial}{\partial t} \log \E \left[ e^{tY} \right]\\
&= 2 \left( \frac{\partial}{\partial \tilde t} \log \E \left[e^{\tilde t \tilde Y}\right]_{|\tilde t = 2t} \right) - 1 \\
&= 2 \left(  \frac{\sum_{k=0}^\infty \left( \prod_{r=0}^k \frac{b + r}{2b + r} \right) \frac{(2t)^k}{k!}}{\sum_{k=0}^\infty \left( \prod_{r=0}^{k-1} \frac{b + r}{2b + r} \right) \frac{(2t)^k}{k!}}  \right) - 1 \\
&=  \frac{J_{b + \frac{1}{2}}(t)}{J_{b-\frac{1}{2}}(t)} ,
\end{split}\end{equation}
where $J_\nu$ is the modified Bessel function of the first order. It has been shown 
that the ratio $J_{\nu +1}/J_\nu$ of modified Bessel functions is concave for every real number $\nu \geq - 1/2$. This is the case since these ratios can be expressed as the pointwise minimum of a set of Amos-type functions \cite[Theorem 11]{GH13}. It can be easily shown that the second derivative of these Amos-type functions is non-positive and hence the ratios $J_{\nu+1}/J_\nu$ are themselves concave since they are the pointwise minimum of concave functions.  

Let $\kappa_j = \frac{\partial^j}{\partial t^j} \log \E \left[ e^{tY}\right]_{| t=0}$. Clearly $\kappa_j = 2^j \frac{\partial^j}{\partial t^j} \log \E \left[ e^{t \tilde Y}\right]]_{| t=0} = 2^j \tilde \kappa_4$ for every $j \geq 2$. Note that $\kappa_4 = \E [(Y - \E Y)^4] - 3 (\text{Var} (Y))^2$ and recall that the \textit{excess kurtosis} is defined as 
\begin{equation}
\text{ExKurt}(Y) = \frac{\E [(Y- \E Y)^4]}{(\text{Var} (Y))^2} - 3.
\end{equation}
Therefore $\kappa_4$ is given by the excess kurtosis of the beta distribution with parameters $a = b >0$ via the equation
\begin{equation}
 \frac{2^{-4} \kappa_4}{(\text{Var} (\tilde Y))^2} = \text{ExKurt}(\tilde Y) = - \frac{6}{3+2b} < 0.
\end{equation}
Hence $\kappa_4$ must be negative. 
\end{proof}

One can also construct examples for which $k>4$:
\begin{pro}\label{step}
Let $\alpha_0$ have the following step density:
\begin{equation}
\alpha_0(d\s) = \frac{d\s}{c} \left\{\begin{array}{ll} 1, & {\rm for\ } |\s| > \frac13, \vspace{.2cm}\\ 
2(59-18\sqrt{10}), & {\rm for\ }  |\s| \leq \frac13,\end{array} \right.
\end{equation}
where $c=8(10-3\sqrt{10})$ is the normalizing constant, and let $g(\s)=\s$. Then, Theorem~\ref{thm-critexp} holds with $k=6$.
\end{pro}
\begin{proof}
For the measure
\begin{equation}\label{eq-defaphab}
\alpha_0(d\s) = \frac{d\s}{c} \left\{\begin{array}{ll} 1, & {\rm for\ } |\s| > \frac13, \vspace{.2cm}\\ 
b, & {\rm for\ }  |\s| \leq \frac13,\end{array} \right.
\end{equation}
with $b>1$, one can compute that 
\begin{equation}
\frac{\partial^4}{\partial t^4}\log \alpha_0(e^{tg})\Big|_{t=0}=-\frac{2(b^2-236 b+964)}{1215 (b+2)^2}.
\end{equation}
Since we want this to equal $0$ in order to have $k>4$,  we need to choose
\begin{equation}
b = 2(59 \pm 18\sqrt{10}).
\end{equation}
Furthermore, we need that, for all $t>0$,
\begin{equation}\label{eq-concavitycondstep}
\frac{\partial^3}{\partial t^3} \log \alpha_0(e^{t g}) <0.
\end{equation}
Making a plot for these two values of $b$ it becomes clear that this is not satisfied for $b = 2(59 + 18\sqrt{10})$. We now prove that this is satisfied for the other value of $b$.

With $\alpha_0$ as in~\eqref{eq-defaphab},
\begin{align}
\log \alpha_0(e^{tg}) &= \log\left(e^t-e^{-t} + (b-1)(e^{\frac13 t}-e^{-\frac13 t} )\right)-\log t-\log c \nn\\
&= \log\left(e^{2t}-1 + (b-1)(e^{\frac43 t}-e^{\frac23 t} )\right)-t- \log t-\log c.
\end{align}
Making the change of variables $z=e^{\frac23t}$, we can write this as 
\begin{equation}
\log \alpha_0(e^{tg}) =  \log\left(z^3-1 + (b-1)(z^2-z)\right)-\frac32 \log z- \log\log z+\log\frac32-\log c.
\end{equation}
Defining the polynomials
\begin{equation}
p_n(z) = 3^n z^3 + 2^n (b-1) z^2 - (b-1) z - \indic{n=0},
\end{equation}
we can write
\begin{equation}
\log\left(z^3-1 + (b-1)(z^2-z)\right) = \log p_0(z).
\end{equation}
Note that
\begin{equation}
z p'_n(z) = z ( 3^{n+1}z^2+2^{n+1}(b-1)z-(b-1)) = p_{n+1}(z),
\end{equation}
and that
\begin{equation}
\frac{\partial}{\partial t} z(t) = \frac23 e^{\frac23 t} = \frac23  z.
\end{equation}
Hence,
\begin{align}
\frac{\partial^3}{\partial t^3} \log p_0(z) &= \frac23 \frac{\partial^2}{\partial t^2} \frac{p_1(z)}{p_0(z)} = \left( \frac23\right)^2 \frac{\partial}{\partial t}\left(\frac{p_2(z)}{p_0(z)} - \left(\frac{p_1(z)}{p_0(z)}\right)^2\right) \nn\\
&=  \left( \frac23\right)^3 \left(\frac{p_3(z)}{p_0(z)} -3 \frac{p_2(z) p_1(z)}{p_0^2(z)}+2 \left(\frac{p_1(z)}{p_0(z)}\right)^3\right).
\end{align}
One can also show that
\begin{equation}
\frac{\partial^3}{\partial t^3}\left(-\frac32 \log z- \log\log z+\log\frac32-\log c\right) = -\frac{16}{27 \log^3(z)}.
\end{equation}
Hence,
\begin{align}
\frac{\partial^3}{\partial t^3} \log \alpha_0(e^{t g}) &= \left( \frac23\right)^3 \left(\frac{p_3(z)}{p_0(z)} -3 \frac{p_2(z) p_1(z)}{p_0^2(z)}+2 \left(\frac{p_1(z)}{p_0(z)}\right)^3\right) -\frac{16}{27 \log^3(z)} \nn\\
&=\left( \frac{2}{3p_0(z)}\right)^3 \left(p_3(z)p_0^2(z) -3 p_2(z) p_1(z)p_0(z)+2 p_1^3(z) -2\frac{p_0^3(z)}{ \log^3(z)}\right).
\end{align}
Since,
\begin{equation}
p_0(z) = (z-1)(z^2+bz+1) >0,
\end{equation}
for all $z>1$, it remains to show that, for all $z>1$,
\begin{equation}\label{eq-condPsmall1}
P(z) := p_3(z)p_0^2(z) -3 p_2(z) p_1(z)p_0(z)+2 p_1^3(z) < 2\frac{p_0^3(z)}{ \log^3(z)}.
\end{equation}
When $P(z)\leq 0$, we are done. One can analyze the zeros of this polynomial for ${b=2(59 -18\sqrt{10})}$, e.g.\ using Mathematica, which shows that this degree 8 polynomial has three zeros smaller than one, four imaginary zeros and one zero at $z=z_0\approx 12.254$. By checking that, e.g., $P(13)<0$, this means that $P(z)\leq0$ for all $z\geq z_0$. It hence remains to analyze $z<z_0$ for which $P(z)>0$.

For $P(z)>0$~\eqref{eq-condPsmall1} is equivalent to
\begin{equation}
\log^3 z - 2 \frac{p_0^3(z)}{P(z)}<0.
\end{equation}
Since,
\begin{equation}
\log^3 z - 2 \frac{p_0^3(z)}{P(z)}\Big|_{z=1} = 0,
\end{equation}
we can write
\begin{equation}
\log^3 z - 2 \frac{p_0^3(z)}{P(z)} = \frac{d}{dz}\left(\log^3 z - 2 \frac{p_0^3(z)}{P(z)}\right)\Big|_{z=\xi} (z-1),
\end{equation}
for some $\xi\in(1,z)$, where it should be noted that also $\xi<z_0$ if $z<z_0$. It hence suffices to show that 
\begin{equation}
z \frac{d}{dz}\left(\log^3 z - 2 \frac{p_0^3(z)}{P(z)}\right) < 0,
\end{equation}
holds for all $1<z<z_0$. The same argument can be repeated twice more to obtain that it suffices to show that, for all $1<z<z_0$,
\begin{equation}
6 - \frac{Q(z)}{P^4(z)} < 0,
\end{equation}
for some polynomial $Q(z)$, which is equivalent to showing that
\begin{equation}
6 P^4(z) - Q(z) < 0.
\end{equation}
We can again analyze the zeros of this degree $33$ polynomial using Mathematica. This shows that there are $18$ zeros at most $1$, 14 imaginary zeros and one zero at $z\approx 42.485$ which is clearly bigger than $z_0$. Computing, e.g., that $6 P^4(2) - Q(2)<0$ shows that $6 P^4(z) - Q(z)<0$ for all $1<z<z_0$, so that we can conclude~\eqref{eq-concavitycondstep}.

Finally,
\begin{equation}
\frac{\partial^6}{\partial t^6}\log \alpha_0(e^{tg})\Big|_{t=0} = -\frac{4(20+7\sqrt{10})}{8505}<0,
\end{equation}
so that indeed $k=6$ is the smallest natural number such that~\eqref{eq-smallestknegative} holds.
\end{proof}

\subsection{State space $S^1$} 

Take the rotation-invariant kernel 
\begin{equation}
K(\s,\s')= e^{\Phi(\s,\s')}=c+ \theta \cos(\s-\s' )=
c+  \theta  [\cos(\s) \cos(\s' )+ \sin(\s) \sin(\s' )],
\end{equation}
whose
range is spanned by the constants and the functions $ \s \mapsto \cos(\s)$ and $\s \mapsto \sin(\s)$. We assume that the constants $c$ and $\theta$ are chosen s.t.\ $K$ is strictly positive. 
Here $\a$ is any probability measure on $S^1$. 

Observe that 
\begin{equation}\begin{split}
T_{\Phi,P,\alpha}(V)(\s)&=\int P^{sb}(dw )\int\nu^{w, V}(d\tilde \s)(K(\s,\tilde \s))\cr
&= c + \theta \cos(\s) \E^{sb}\int\nu^{W V}(d\tilde \s) \cos (\tilde\s)
+ \theta \sin(\s) \E^{sb}\int\nu^{W V}(d\tilde\s) \sin(\tilde \s).
\end{split}\end{equation}
Hence any $V$ satisfying 
\begin{equation}\begin{split}
V=T_{\Phi,P,\alpha}(V),
\end{split}\end{equation}
must be of the form 
$V(\s)=c +  b\cos(\s) + a \sin(\s)$. Hence the fixed point equation reduces to 
the two-dimensional equation 

\begin{equation}\begin{split}
&b/\theta= \E^{sb}\nu^{W (b\cos(\cdot) + a \sin(\cdot))}(\cos(\s))=:F_1(b,a),\cr
&a/\theta= \E^{sb}\nu^{W (b\cos(\cdot) + a \sin(\cdot))}(\sin(\s))=:F_2(b,a).\cr
\end{split}\end{equation}

Suppose that there is an $S^1$-rotational symmetry also of the measure $\a$. 
Then one solution $(a,b)$ can be rotated and gives a whole orbit of solutions. 
It is therefore enough to choose one representative and 
put $a=0$. Then the second equation is automatically fulfilled. 

We repeat the calculation in the presence of a field, which is added by introducing the a priori measures 
$\a_h$ with $\frac{d\a_h}{d\a}=e^{h \cos(\cdot)}/ \a(e^{h \cos(\cdot)})$, 
that is we have a magnetic field which couples to the cosine.  

Using Theorem~\ref{thm-pressure} we rewrite the free energy in the following form, 
where the supremum is over all solutions $(a,b)$ of the fixed point equation, 
\begin{equation}
\begin{split}
\psi(\Phi,P,h) &= \sup_{a,b} \left(\frac{c}{2}\E(W) -\frac{1}{2}\E(W) \frac{b^2+a^2}\theta + 
\E \log \alpha_h(e^{W( b\cos(\cdot) + a \sin(\cdot))})\right)\cr
&=: \frac{c}{2}\E(W)+\sup_{a,b} \chi_h(b,a).
\end{split}
\end{equation} 

As mentioned above, for vanishing magnetic field $h$ it suffices to take $a=0$. 
For $h>0$, we can reparametrize $\chi_h(b,a)$ by setting $b = \tilde{b} \cos(\phi_0)$ and $a = \tilde{b} \sin(\phi_0)$. Then
\begin{equation}
\psi(\Phi,P,h) = \frac{c}{2}\E(W)+\sup_{a,b} \chi_h(b,a) = \frac{c}{2}\E(W)+\sup_{\tilde{b}} \sup_{\phi_0} \tilde\chi_h(\tilde{b},\phi_0), 
\end{equation} 
where
\begin{equation}
\chi_h(\tilde{b},\phi_0)  = -\frac{1}{2}\E(W) \frac{\tilde{b}^2}\theta + 
\E \log \alpha_h(e^{W \tilde{b} \cos(\cdot-\phi_0)}).
\end{equation} 
Now observe that, for all $h,t>0$,
\begin{equation}
 \alpha_h(e^{t \cos(\cdot-\phi_0)}) = \frac{ \alpha_0(e^{h \cos(\cdot) + t \cos(\cdot-\phi_0)})}{ \alpha_0(e^{h\cos(\cdot)})}.
\end{equation}
We can reparametrize again, setting $h+t \cos(\phi_0)=\tilde{h} \cos(\psi_0)$ and $t \sin(\phi_0)=\tilde{h} \sin(\psi_0)$. Then,
\begin{equation}
\alpha_0(e^{h \cos(\cdot) + t \cos(\cdot-\phi_0)}) =  \alpha_0(e^{\tilde{h} \cos(\phi-\psi_0)}),
\end{equation}
which is independent of $\psi_0$ by rotation symmetry. It thus suffices to optimize $\tilde{h}=\tilde{h}(\phi)$. Since,
\begin{equation}
\tilde{h}^2 = (h+\cos(\phi_0))^2+\sin(\phi_0)^2,
\end{equation}
this is clearly maximized by choosing $\phi_0=0$.
Hence, also in the presence of a field it suffices to only look at solutions for which $a=0$.

So our problem is reduced to 
\begin{equation}\begin{split}
&b/\theta= \E^{sb}\nu^{W b\cos(\cdot) }(\cos(\s)). \cr
\end{split}\end{equation}
We have already seen in the proof of Proposition \ref{pro:beta} that the map $[0,\infty)\ni t \mapsto \nu^{t \cos}(\cos)= \frac{J_1(t)}{J_0(t)}$ 
is concave. So the phase transition is second order here, too, with critical $\theta$ 
given by 
\begin{equation}
1/\theta_c=\E^{sb}(W)\a(\cos^2(\s))=\frac{\E^{sb}(W)}{2}.
\end{equation}

\subsection{State space $S^q$ \label{sec-criticalbehavior5.4}} 
Take the rotation-invariant kernel $K(\s,\s') = c+ \theta \s \cdot \s'$ where we assume that the constants $c$ and $\theta$ are chosen s.t.\ $K$ is strictly positive.
Observe that 
\begin{equation}\begin{split}
T_{\Phi,P,\alpha}(V)(\s)&=\int P^{sb}(dw )\int\nu^{w, V}(d\tilde \s)(K(\s,\tilde \s))\cr
&= c + \theta \s \cdot \E^{sb}\int\nu^{W V}(\tilde\s).
\end{split}\end{equation}
Hence any $V$ satisfying $V=T_{\Phi,P,\alpha}(V)$
must be of the form $V(\s)=c +  m\cdot\s$, where $m$ must satisfy
\begin{equation}
\frac{m}{\theta}=\E^{sb} [\nu^{W, m\cdot \tilde\s}(\tilde \s)] = \E^{sb}\left[\frac{\alpha_h(e^{Wm\cdot\tilde\s}\tilde\s)}{\alpha_h(e^{Wm\cdot\tilde\s})}\right].
\end{equation}
For $h=0$ every rotation of solutions of this equation is also a solution, so w.l.o.g.\ we can assume that $m=|m|e_1$. Then, we get
\begin{equation}
\frac{|m|}{\theta}e_1= \E^{sb}\left[\frac{\alpha_0(e^{W|m|e_1\cdot\tilde\s}\tilde\s)}{\alpha_0(e^{W|m|e_1\cdot\tilde\s})}\right].
\end{equation}
For $h\neq 0$, we have by Theorem~\ref{thm-pressure} that
\begin{equation}\begin{split}
\psi(\Phi,P) &= \sup_m \left(-\frac{\E[W]}{2} \E^{sb} [\nu^{W, m\cdot \s}(c+m \cdot \s)]+\E [\log \alpha_h(e^{W(c+m\cdot\s)})]\right) \cr
&= \sup_m \left(\frac{\E[W]c}{2}-\frac{m\cdot m}{\theta}+\E [\log \alpha_h(e^{Wm\cdot\s})]\right) \cr
&= \frac{\E[W]c}{2} + \sup_{|m|} \left(-\frac{|m|^2}{\theta}+\sup_{\tilde m \,:\, |\tilde m|=|m|}  \E [\log \alpha_h(e^{W\tilde m\cdot\s})]\right).
\end{split}\end{equation}
Since,
\begin{equation}
\alpha_h(e^{t\cdot\s}) = \frac{\alpha_0(e^{t\cdot\s+h\cdot\s})}{\alpha_0(e^{t\cdot\s})},
\end{equation}
we have to maximize $\alpha_0(e^{(t+h)\cdot\s})$ for fixed $|t|$. By rotation symmetry this is a monotone increasing function of $|t+h|$. We hence need to maximize $|t+h|$ for fixed $h$ and $|t|$. That is, we need to maximize
\begin{equation}
|t+h| = \sqrt{|t|^2 + h\cdot t + |h|^2},
\end{equation}
so that it suffices to maximize $h\cdot t$ for given $|t|$ and $h$. This is clearly done by taking
\begin{equation}
t=\frac{|t|}{|h|}h.
\end{equation}
If we assume, w.l.o.g., that $h=|h|e_1$, we hence get that
\begin{equation}
\psi(\Phi,P,\a) = \frac{\E[W]c}{2} + \sup_{|m|} \left(-\frac{|m|^2}{\theta}+ \E [\log \alpha_h(e^{W |m|e_1\cdot\s})]\right),
\end{equation}
where the supremum is over solutions to
\begin{equation}
\frac{|m|}{\theta} e_1 =  \E^{sb}\left[\frac{\alpha_h(e^{W|m|e_1\cdot\s}\s)}{\alpha_h(e^{W|m|e_1\cdot\s})}\right].
\end{equation}
Suppose now that $\alpha_0$ is the uniform distribution over $S^q$. Then, $e_1 \cdot \s = \cos\varphi$, where $\varphi$ is the angle between $e_1$ and $\s$. The measure of such a given angle is, for $\varphi\in[-\pi/2,\pi/2]$, given by
\begin{equation}
\alpha(d\varphi) =c (\sin \varphi)^{q-1} d\varphi,
\end{equation}
for some normalization constant $c$. Hence,
\begin{equation}
\alpha_0(e_1 \cdot \s \leq t) =\alpha_0(\cos\varphi  \leq t) = \alpha_0(\varphi  \geq \arccos t) = c \int_{\arccos t}^{\pi/2} (\sin \varphi)^{q-1} d\varphi,
\end{equation}
since $\arccos$ is a decreasing function. Hence,
\begin{equation}\begin{split}
\alpha_0(dt) & = \frac{d}{dt} c \int_{\arccos t}^{\pi/2} (\sin \varphi)^{q-1} d\varphi = -(\sin(\arccos t)^{q-1} \frac{d}{dt} \arccos t \cr
&= (1-t^2)^{q/2-1} = (1-t)^{q/2-1}(1+t)^{q/2-1}.
\end{split}\end{equation}
Hence, this model on $S^q$ behaves the same as the model with spins on $[-1,1]$, where $\alpha_0$ has a Beta distribution with $\beta=q/2$.

\subsection*{Acknowledgements.}
This work was supported by DFG Research Training Group 2131: High-dimensional Phenomena in Probability -- Fluctuations and Discontinuity.

\end{document}